\documentclass{gen-j-l}

\usepackage{graphicx}
\usepackage{enumitem}
\usepackage{gensymb}
\usepackage{csquotes}
\usepackage{amssymb}
\usepackage{calc}
\usepackage{mdframed}
\usepackage{xcolor}

\mdfdefinestyle{greybox}{linewidth=0,backgroundcolor=black!10}

\numberwithin{figure}{subsection}

\newtheorem{theorem}{Theorem}[subsection]
\newtheorem{lemma}[theorem]{Lemma}
\newtheorem{conjecture}[theorem]{Conjecture}

\theoremstyle{definition}
\newtheorem{definition}[theorem]{Definition}
\newtheorem{method}[theorem]{Method}
\newtheorem{fact}[theorem]{Fact}
\newtheorem{problem}[theorem]{Problem}
\newtheorem{example}[theorem]{Example}

\theoremstyle{remark}

\DeclareMathOperator{\Imb}{Imb}
\DeclareMathOperator{\imb}{imb}
\DeclareMathOperator{\dist}{dist}
\DeclareMathOperator{\inj}{inj}

\newcommand{\degrees}{\degree}

\newenvironment{itemize*}
  {\begin{itemize}[topsep=-\parskip+\jot]}
  {\end{itemize}}
  
\newenvironment{enumerate*}
  {\begin{enumerate}[label=(\alph*),topsep=-\parskip+\jot]}
  {\end{enumerate}}
  
\begin{document}

\title{Geodesic nets: some examples and open problems}

\author{Alexander Nabutovsky \and Fabian Parsch}
\address{Department of Mathematics, University of Toronto, 40 St. George Street, Toronto, ON M5S 2E4, Canada}
\email{alex@math.toronto.edu, fparsch@math.toronto.edu}
\copyrightinfo{2019}{Alexander Nabutovsky and Fabian Parsch}

\begin{abstract}
Geodesic nets on Riemannian manifolds form a natural class of stationary objects generalizing geodesics. Yet almost nothing is known about their
classification or general properties even when the ambient Riemannian
manifold is the Euclidean plane or the round $2$-sphere. 

In the first half of this paper we survey some results and open questions (old and new) about geodesic
nets on Riemannian manifolds. Many of these open questions are about geodesic nets on the Euclidean plane. The second half contains a 
partial answer for one of these questions, namely, a description of a new infinite family of geodesic nets on the Euclidean plane with 14 boundary (or unbalanced) vertices and arbitrarily many inner (or balanced) vertices of degree $\geq 3$. 
\end{abstract}

\maketitle

%%%%%%%%%%%%%%%%%%%%%%%%%%%%%%%%%%%%%
\section{Overview}
%%%%%%%%%%%%%%%%%%%%%%%%%%%%%%%%%%%%%
\label{sect:overview}

\subsection{Geodesic nets and multinets: definition} Let $M$ be a Riemannian manifold, $S$ a finite (possibly empty)
set of points in $M$, and $G$ a finite multigraph (or, more formally,
a finite $1$-dimensional cell complex).
A {\it geodesic net modelled on} $G$ {\it with vertices} $S$ is a smooth embedding $f$ of $G$ into $M$ such that:
\begin{enumerate}
\item Every point from $S$ is the image under $f$ of a vertex of $G$;
\item For each $1$-parametric flow $\Phi_t,\ t\in(-\epsilon,\epsilon),$ of diffeomorphisms of $M$ fixing all points of  $S$ with $\Phi_0=Id$, $t=0$ is the critical point of
the function $l(t)$ defined as the length of $\Phi_t(f(G))$.
\end{enumerate}

Less formally, geodesic nets on $M$ are critical points (not necessarily local minima!) of the length functional on the space of embedded multigraphs into $M$, where a certain
subset of the set of vertices must be mapped
to prescribed points of $M$.

The simplest example of geodesic nets arises when $S$ is a set of
two points $x,y$, and $G=[0,1]$ the graph with two vertices and one edge. In this case, the geodesic nets modelled on $G$ with vertices
$x,y$ are precisely non-self-intersecting geodesics in $M$ connecting $x$ and $y$. Self-intersecting geodesics can be modelled on more 
complicated graphs, but if we wish to model them on the same $G$ it
makes sense to modify our definition of geodesic nets by allowing
$f$ to be only an immersion on the union of interiors of edges.
In other words, one may allow edges to self-intersect and to intersect each other. In particular, we are allowing that two different
edges between the same pair of vertices might have the same image.
As the result, 
the images of edges of $G$ in $M$ acquire
multiplicities that can be arbitrary positive integer numbers.
In this paper we are going to call geodesic nets defined
using immersions rather than embeddings of multigraphs {\it geodesic multinets}.

Applying the first variation formula for for the length functional
we see that the above definition of a geodesic (multi)net is equivalent to the following:

\begin{definition}
Let $S$ be a (possibly empty) finite set of points in a Riemannian manifold $M$.
A geodesic net on $M$ consists of a finite set $V$ of points
of $M$ (called vertices) that includes $S$ and a finite set $E$ of non-constant distinct geodesics between vertices (called edges) so that for every
vertex $v\in V\setminus S$ the following balancing condition
holds:
Consider the unit tangent vectors at $v$ to all edges incident
to $v$. Direct each tangent vector from $v$ towards the other endpoint 
of the edge. Then the sum of all these tangent vectors must be equal
to $0\in T_vM$. Further, edges are not allowed to intersect or self-intersect.
Geodesic multinets are defined in the same way with the two following
distinctions: 1) Edges are allowed to intersect and self-intersect;
2) Each edge is endowed with a positive integer multiplicity;
the tangent vector to an edge enters the sum
in the balancing condition at each its endpoint with the
multiplicity equal to the multiplicity of the corresponding edge.
\end{definition}

Vertices in $S$ are called {\it boundary} or {\it unbalanced},
vertices in $V\setminus S$ are called {\it inner}, or {\it free}, or {\it balanced}
(as the balancing condition must hold only at each vertex in $V\setminus S$). If $v$ is an unbalanced vertex, then the sum of all unit tangent vectors to edges incident to $v$ need not be
equal to the zero vector. We call this sum {\it the imbalance
vector} $\Imb(v)$ at $v$, and its norm {\it the imbalance}, $\imb(v)$, at $v$. The sum of imbalances $\sum_{v\in S}\imb(v)$ over the set of all unbalanced
points is called {\it the total imbalance} of the geodesic net. It is
convenient to define $\Imb(v)$ also at balanced vertices as zero vectors
in $T_vM$. 

For the rest of the paper we are going to require that no balanced vertex is isolated,
that is, has degree zero. As the degree of a balanced vertex
clearly cannot be one, we see that the minimal degree of a balanced
vertex becomes two. The balancing condition implies that for any balanced vertex of degree $2$, its two incident edges can be merged into a single geodesic. Conversely, we can subdivide each edge of
a geodesic net by inserting as many new balanced vertices of degree $2$
as we wish. As now the role of balanced vertices of degree $2$ in the
classification of geodesic nets is completely clear, we are going
to consider below only geodesic nets where {\bf all balanced vertices
have degree $\geq 3$}. It is clear that we can add or remove
geodesics connecting unbalanced vertices at will without affecting
the balancing condition at a balanced vertex. Therefore, we agree that
all considered geodesic nets {\bf do not contain edges between
unbalanced vertices.}
Our final convention is that we are going to consider here only
{\bf connected geodesic (multi)nets} (as the classification
of disconnected nets obviously reduces to classification of their
connected components).

\subsection{Geodesic nets in Euclidean spaces.}
A significant part of this paper will be devoted to geodesic nets
in Euclidean spaces, in particular, in the Euclidean plane. In this
case above definition says that a geodesic multinet is a graph $G=(V,E)$ in $\mathbb{R}^n$ such that 1) $S$ is a subset of the set of vertices $V$, 2) each edge
is a straight line segment between its endpoints, and is
endowed with a positive integer multiplicity $n(e)$, 3) For each
vertex $v\in V\setminus S$, there is zero imbalance, i.e. $\sum_{e\in I(v)}n(e)\frac{e}{\Vert e\Vert}=0$, where $I(v)$ denotes the set of edges incident
to $v$, and $e$ denotes an edge regarded as the vector in $\mathbb{R}^n$ directed from $v$ towards the other endpoint.
For geodesic nets, all $n(e)$ must be equal to $1$, and, in addition,
different edges are not allowed to intersect.
Figure \ref{fig:balancedvertices} depicts examples of 
balanced points of degrees $3$, $4$ and $7$ in $\mathbb{R}^2$.
It is easy to see that 1) the angles between edges incident to
a balanced vertex of degree $3$ are always equal to $120$ degrees.
(This will be true not only for $\mathbb{R}^2$ but for all ambient
Riemannian manifolds $M$.) 2) A balanced vertex $v$ of degree $4$
in the Euclidean plane is a point of intersection of two straight line
segments formed by two pairs of incident edges at $v$ (see Figure \ref{fig:balancedvertices}).

\begin{figure}
	\centering
	\includegraphics{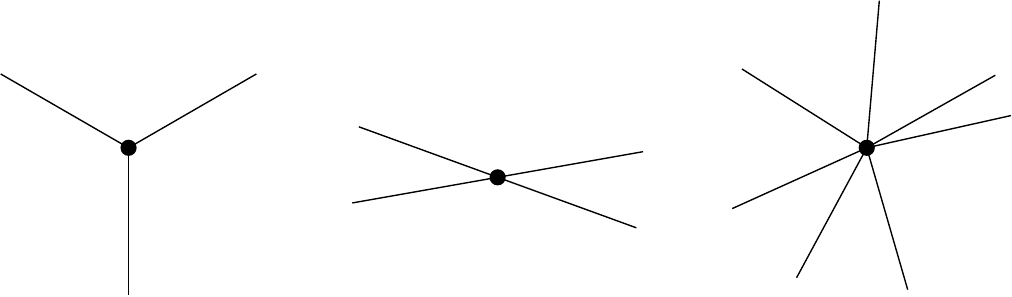}
	\caption{Examples for balanced vertices of degree $3$, $4$ and $7$.}
	\label{fig:balancedvertices}
\end{figure}

Here are some other easily verified facts about geodesic (multi)nets in
Euclidean spaces:

\begin{enumerate}
\item Each geodesic multinet is contained in the convex hull
of its unbalanced vertices. 
\item As a corollary each geodesic (multi)net
with two boundary vertices is simply the straight line segment connecting
these points (that can be endowed with any positive integer multiplicity in the case of geodesic multinets).
\end{enumerate}

Therefore, the interesting part of the classification of geodesic (multi)nets in the Euclidean plane starts from the case of three boundary points. 

\begin{enumerate}
\setcounter{enumi}{2}
\item For each geodesic (multi)net in $\mathbb{R}^n$ we can consider
$\Imb(v)$ as vectors in the ambient space $\mathbb{R}^n$. 
Therefore, in this case one can also define the total
imbalance vector. Yet this vector is always zero:
$$\sum_{v\in S}\Imb(v)=\sum_{v\in V}\Imb(v)=0,\ \ \ (1.2.1)$$

\par\noindent
Indeed, the second sum can be represented as the sum 
of contributions of individual edges. Each edge contributes
two oppositely directed vectors that enter sums in the definition
of imbalance vectors at its endpoints. Therefore, the sum over edges of edge contributions is zero.

\item For a geodesic (multi)net $N$ in a Euclidean space its length
$L(N)$ is given by the following formula:
$$L(N)=-\sum_{v\in S}\langle v , \Imb(v)\rangle.\ \ \ (1.2.2)$$

In the right hand side we perform the summation over the set of
all unbalanced vertices; each vertex is also regarded as a
vector in $\mathbb{R}^n$.\\ 
{\bf Proof:} In order to prove this formula, first
observe that the right hand side does not change when we change
the origin of the coordinate system in $\mathbb{R}^n$. (This
easily follows from the formula (1.2.1).) Therefore, we can
assume that the origin is not on the net. For each positive
$r$ let $D_r$ denote the ball of radius $r$ centered at the origin, $\partial D_r$ its boundary, $E(r)$ the set of edges of the
net intersecting $\partial D_r$. For each $e\in E(r)$ let $e(r)$ denote
the point of intersection of $e$ and $\partial D_r$. Formula (1.2.2)
is an immediate corollary of the following formula, when it is applied to very large values of $r$:
$$L(N\bigcap D_r)=\sum_{e\in E(r)}\langle e(r), \frac{e}{\Vert e\Vert}\rangle-\sum_{v\in D_r}\langle v, \Imb(v)\rangle.\ \ \ (1.2.3)$$
In this formula we regard $e$ also as a vector in $\mathbb{R}^n$. We
choose its direction from $e(r)$ towards the interior of $D_r$.
This formula obviously holds, when $r$ is small, as both sides
are equal to zero. Define special values of $r$ as those, where
$\partial D_r$ is either tangent to one of the edges or
passes through one of the vertices. There are only finitely
many special values of $r$. Our next observation which is easy
to verify is that the right hand side of (1.2.3) changes continuously, when $r$ passes through its special value. (Obviously, one needs only to check what happens if $\partial D_r$ passes through a balanced or an unbalanced vertex.)
Now we see that it is sufficient to check that the derivatives
of the right hand side and the left hand side with respect to $r$ at each non-special point coincide. Each of these derivatives
will be a sum over edges in $E(r)$. To complete the proof it is sufficient to verify that the contributions of each edge to both sides are the same. Each such edge $e$ contributes
$\frac{1}{\cos\theta_e(r)}$ to the derivative of the left hand
side, where $\theta_e(r)$ denotes the angle between $e$ and $e(r)$. Its contribution to the right hand side is $(r\cos \theta_e(r))'=\cos\theta_e(r)-r\sin\theta_e(r)\frac{d\theta_e(r)}{dr}=\frac{1}{\cos\theta_e(r)},$ as an easy trigonometric argument implies that $\frac{d\theta_e(r)}{dr}=-\frac{\tan\theta_e(r)}{r}.$
This completes the proof of (1.2.3) and, therefore (1.2.2).
\end{enumerate}

\subsection{Steiner trees and locally minimal geodesic nets.}

\begin{figure}
	\centering
	\includegraphics{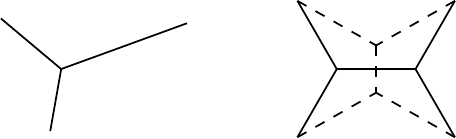}
	\caption{A geodesic net with $3$ unbalanced vertices and $1$ balanced vertex (the \textit{Fermat Point}). This is in fact the maximal number of balanced vertices when only given $3$ unbalanced vertices on the plane with a metric of nonpositive curvature. On the other hand, the two Steiner trees for four points are not maximal regarding the number of balanced vertices of a geodesic net with four unbalanced vertices.}
	\label{fig:threeandfournet}
\end{figure}

Study of geodesic nets was originally
motivated by the following question posed by Gau\ss: Given a set of
points on the plane, connect them by means of a graph of the minimal
possible length. It is easy to see that this graph is always a 
geodesic net modelled on a tree
(called the Steiner tree). This tree is a geodesic net, where
the given points are unbalanced points, but typically it also
contains new balanced vertices. It is easy to prove that all balanced
vertices of a Steiner tree have degree $3$. The first and the most
fundamental example is the case of three points $A$, $B$, $C$ on the plane
forming a triangle with angles $<120\degrees$. In this case there
exists the (unique) point $O$ in the triangle $ABC$ called the
Fermat point, such that the angles $AOB$, $BOC$ and $COA$ are all
equal to $120\degrees$. The Steiner tree will consist of three edges
$OA$, $OB$, and $OC$ (see Figure \ref{fig:threeandfournet}). The Steiner tree 
on four given (unbalanced) vertices might involve two extra (balanced)
vertices (see same figure). A geodesic net (in a Riemannian manifold) is called 
{\it locally minimal} if its intersections with all sufficiently small
balls are Steiner trees (for the set of points formed by all intersection of
the geodesic net with the boundary circle and all unbalanced points in the circle). For geodesic nets in the Euclidean plane
the local minimality is equivalent to the requirement that all balanced
points have degree $3$. The locally minimal geodesic nets in Euclidean
spaces and, more generally, Riemannian manifolds were 
extensively investigated
by A. Ivanov and V. Tuzhilin (cf. \cite{ivanovtuzhilinbook}, \cite{ivanovtuzhilinreview}). (Note that although general
geodesic net are not locally minimal with respect to this condition,
they are locally minimal in the following less restrictive sense: For each point $p$ on the net
and all sufficiently small $r$ the intersection of the net with the
ball of radius $r$ provides the global minimum of the length
among all trees {\it of the same shape} (i.e. star-shaped with the
same number of edges) connecting the boundary points.

The idea of minimization of length might seem useful if one wants
to construct a geodesic net with the set $S$ of boundary points modelled on a given graph $G$, say, in the Euclidean plane as follows:
Consider all embeddings of $G$ in the plane such that all edges 
are mapped into straight line segments, and a certain set of vertices
is being mapped to $S$. Yet the positions of other vertices are
variable, and we do not insist on balancing condition at any vertices.
Now we are going to minimize the total length of all edges of the
graph over the set of such embeddings. It is easy to see that the
total length will be a convex function and has the unique minimum.
Moreover, one can start with an arbitrary allowed embedding of $G$
and use an easy algorithm based on the gradient descent that numerically finds this minimum which will be always a geodesic net.
The problem is that in the process of gradient descent different vertices or edges can merge, and some edges can shrink to a point.
Then thee resulting graph will not be isomorphic to $G$ anymore. In fact, our numerical experiments seem to indicate that if one starts from a random allowed embedding of $G$ one typically ends at very simple geodesic nets such as, for example, the geodesic net with just one extra (balanced) vertex in the centre.

\subsection{Plan of the rest of the paper.} In section \ref{sect:closednets}
we survey closed geodesic nets on closed Riemannian manifolds.
In section \ref{sect:euclidnets} we survey geodesic nets on Euclidean spaces and 
Riemannian surfaces. The emphasize there will be on (im)possibility to majorize the number of balanced points in
terms of the number of unbalanced points ( and possibly, also
the total imbalance). Section \ref{sect:thestar} contains a rather long
construction of an infinite sequence of geodesic nets on the Euclidean plane with $1$ boundary vertices and arbitrarily
many balanced vertices. This sequence provides a partial answer for one
of the questions asked in section \ref{sect:euclidnets}.

\section{Closed geodesic nets.}
\label{sect:closednets}
Geodesic nets with $S=\emptyset$ are called {\it closed geodesic nets}.
The simplest examples of closed geodesic nets are periodic geodesics
(that can be modelled on any cyclic graph or the multigraph with one vertex and one loop-shaped edge) or, more generally, unions of periodic geodesics. The simplest example of a closed geodesic net not containing
a non-trivial periodic geodesic is modelled on the $\theta$-graph
with $2$ vertices connected by $3$ distinct edges. The corresponding
closed geodesic net consists of two vertices connected by $3$ distinct geodesics, so that all angles between each pair of geodesics at
each of the vertices are equal to $120\degrees$. J. Hass and F. Morgan (\cite{hassmorgan}) proved that for each convex Riemannian $S^2$ sufficiently
close to a round metric there exists a closed geodesic net modelled on
the $\theta$-graph. It is remarkable that this is the only known result
asserting the existence of closed geodesic nets not composed of periodic geodesics on an open (in $C^2$ topology) set of Riemannian metrics on a closed manifold!

\begin{problem}
Is it true that each closed Riemannian manifold
contains a closed geodesic multinet not containing a non-trivial periodic
geodesic?
\end{problem}

The standard Morse-theoretic approach to constructing periodic geodesics fails when applied to constructing closed geodesic nets, as any gradient-like flow might make the underlying
multigraph to collapse to a (possibly mutiple) closed curve and, thus, yields only a periodic geodesic.

A classification of shapes of closed geodesic nets on specific closed
Riemannian surfaces is aided by the Gau\ss-Bonnet theorem and the
obvious observation that if a geodesic net on, say, a Riemannian
$S^2$ is modelled on a graph $G$, then $G$ must be planar. Using
these observations A. Heppes (\cite{heppes}) classified all closed geodesic nets
on the round $S^2$, where all vertices have degree $3$ (there are just
nine possible shapes). 
On the other hand, we are not
aware of any restrictions on shapes of closed Riemannian manifolds
of dimension $>2$.

The first question one might ask about closed geodesic nets in Riemannian manifolds
of dimension $\geq 3$ is the following:

\begin{problem}
Classify all $3$-regular graphs $G$ such that the round
$3$-sphere has a geodesic net modelled on $G$.
\end{problem}

Another reasonable question (which, of course, can also be
asked for surfaces) is:

\begin{problem}
Is it true that each closed Riemannian manifold of dimension
$\geq 3$ has a $\theta$-graph shaped closed geodesic net?
\end{problem}

To the best of our knowledge, nothing else is known about classification
of geodesic nets on round $S^2$. In particular, the answer for the
following problem posed by Spencer Becker-Kahn (\cite{becker}) is not known even
when $M$ is the round $2$-sphere.

\begin{problem}[Becker-Kahn]\label{problem:beckerkahn} Let $M$ be a closed Riemannian
manifold. Is there a function $f_M$ (depending on geometry
and topology of $M$) such that each closed geodesic net on $M$
of length $L$ has at most $f_M(L)$ (balanced) vertices?
\end{problem}

As we already noticed the set of closed geodesic nets includes periodic
geodesics as well as their unions. Yet the standard ``folk" argument involving
the compactness of the set of closed curves of length $\leq x$ parametrised by the arclength on a closed Riemannian manifold,
and a quantitative (Yomdin-style) version of the Sard-Smale
theorem
that implies that the set of non-constant periodic geodesics on
a generic closed Riemannian manifold is countable,
also implies the set of closed geodesic nets is countable as well.
So, closed geodesic nets are also ``rare". This fact might be at least
partially responsible for the scarcity of examples of closed geodesic
nets not containing  periodic geodesics.

Surprisingly, many extremely hard open problems about periodic
geodesics can be solved when asked about closed geodesic nets.
Here are some results about the existence of 
closed geodesic nets with interesting properties:

\begin{enumerate}
\item One of the authors (A.N.) and R. Rotman proved the
existence of a constant $c(n)$ such that each closed Riemannian
manifold $M^n$ contains a closed geodesic multinet of length $\leq c(n)vol(M^n)^{\frac{1}{n}}$. It also contains a closed geodesic
multinet of length $\leq c(n)$ diameter$(M^n)$ (\cite{NRshapes}). R. Rotman later
improved this result and proved that one can choose a closed
geodesic multinet satisfying these estimates that has a shape of
a flower, that is consist of (possibly) multiple geodesic loops
based at the same point (vertex) (\cite{Rflowers}). (Of course, the balancing
(stationarity) condition at this point must hold.)
\item Recently L. Guth and Y.Liokumovich (\cite{guthliok})
proved that for a generic closed Riemannian manifold the union
of all closed geodesic multinets must be a dense set. 
\end{enumerate}

Note that these results do not shed any light on the existence of closed geodesic nets that do not include any periodic geodesic
on closed manifolds as all closed geodesic
nets in these theorems might be just periodic geodesics.
Yet, in dimensions $>2$ it is completely unknown if either
of the quoted results from \cite{NRshapes} and \cite{guthliok} holds for periodic
geodesics instead of geodesic multinets.

Finally, note that closed geodesic multinets can be a useful tool
to study other minimal objects on general closed Riemannian
manifolds. For example, recently Rotman proved that for
each closed Riemannian manifold $M^n$ and positive $\epsilon$, 
there exists a ``wide" geodesic loop on $M^n$ with the angle
greater than $\pi-\epsilon$ so that its length is bounded only
in terms of $n$, $\epsilon$ and the volume of $M^n$ (\cite{Rwide}). Alternatively,
one can also use the diameter of $M^n$ instead of its volume.
The proof involves demonstrating the existence of closed geodesic
multinets with certain properties. (Yet these nets can turn out to
be a periodic geodesic in which case the short wide geodesic loop
will be a short periodic geodesic as well.)

Our last remark about closed geodesic multinets is that
in some sense they can be considered a better $1$-dimensional analog of minimal surfaces in higher
dimensions than periodic geodesics. Indeed, minimal surfaces
tend to develop singularities. Their existence
is frequently proven through a version of Morse theory on 
spaces of cycles, where the resulting minimal surface
first arises as a stationary varifold. Similar arguments
using the space of $1$-cycles lead to proofs of existence
of closed geodesic multinets that can be regarded as a 
particularly nice class of stationary $1$-varifolds.
We refer the reader to \cite{allardalmgren} for properties of stationary
$1$-varifolds including a version of formula (1.2.2) (``monotonicity formula") valid
for stationary $1$-varifolds, and, therefore, closed
geodesic multinets on Riemannian manifolds.

\section{Geodesic nets and multinets in Euclidean spaces and Riemannian manifolds}
\label{sect:euclidnets}

Recall, that we agreed to consider only connected geodesic
nets with balanced vertices of degree $\geq 3$ and without edges running between unbalanced vertices. (However, it is still possible
that the union of two or more edges forms a segment between two
unbalanced vertices; of course, this segment will not be an edge.)

\subsection{Geodesic nets and multinets on the Euclidean plane
and more general Riemannian surfaces.} We are going to start
from the description of the following example (see Figure \ref{fig:weightedthree}):

\begin{figure}
    \centering
    \includegraphics{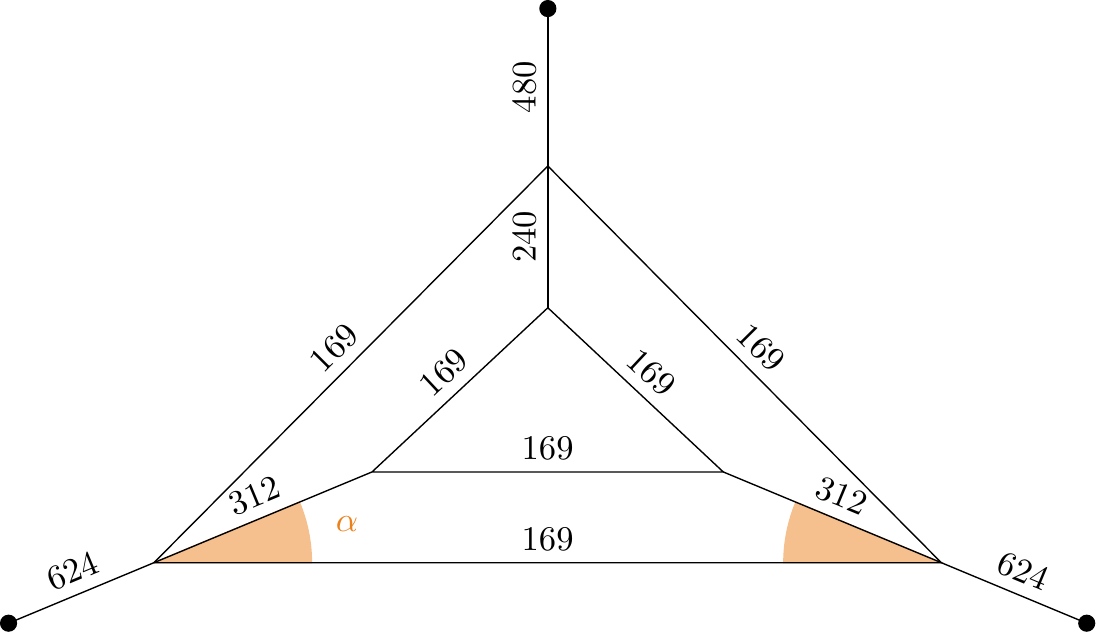}
    \caption{An example of a geodesic \emph{multi}net with three unbalanced vertices and six balanced vertices. Through continuous nesting, the number of balanced vertices can be increased arbitrarily. However, this is at the expense of additional imbalance at the three unbalanced vertices. In this example, $\cos\alpha=12/13$ and $\sin\alpha=5/13$}
    \label{fig:weightedthree}
\end{figure}

\begin{example}\label{ex:weightedthree}
Let $A_1A_2A_3$ be a triangle. Denote its angle at $A_i$ by $\alpha_i$. Assume that for each $i=1,2,3$,
$\cos \frac {\alpha_i}{2}$ is a rational number.
It is easy to produce an infinite set of such
triples of angles using Pythagorean triples of integers.
For example, we can take $\alpha_1=\alpha_2=2\arcsin(\frac{5}{13})$, and $\alpha_3=\pi-\alpha_1-\alpha_2$. Any choice of angles
$\alpha_i$ determines the triangle $A_1A_2A_3$ up to a similarity; the exact choice of its side lengths is not important
for us. As $\cos\frac{\alpha_i}{2}$ is rational, it can be written as $\frac{m_i}{n_i}$ for positive integer $m_i$ and 
$n_i$. Let $N$ denote $n_1n_2n_3$ and $N_i$ denote (integer)
$\frac{m_iN}{n_i}$. Further, let $0<r_1<r_2<\ldots <r_k<1$ be
any finite increasing sequence of positive numbers $<1$ and
$O$ denote the point of intersection of bisectors of angles
$\alpha_i$.
The set of vertices of a geodesic multinet that we are going
to describe looks as follows: It has three unbalanced vertices
$A_1,A_2$ and $A_3$. To describe its set of balanced vertices
consider $k$ homotheties of $A_1A_2A_3$ with center $O$ using
ratios $r_1,\ldots , r_k$. Denote the corresponding vertices
of the homothetic triangles by $A_1^jA_2^jA_3^j$, $j\in\{1,\ldots ,k\}$.
The set of balanced vertices of the geodesic multinet will
include all vertices $A_i^j$. Observe that for each
$i=1,2$ or $3$ vertices $A_i^j$ will subdivide $A_i^1A_i$ into
$k$ segments. We are going to denote these segments by $e_i^j$,
$j=1,\ldots , k$ where the numeration by superscripts $j$ goes
in the order from $A_i^1$ to $A_i$, so that $e_i^1=A_i^1A_i^2$ and
$e_i^k=A_i^kA_i$. All these segments $e_i^j$ will be edges
of the geodesic multinet; the weight of $e_i^j$ will be
equal to $2jN_i$. The set of edges of the multinet will
also include all sides of the triangles $A_i^jA_2^jA_3^j$; all
these edges will be endowed with the same weight $N$.
(Of course, we can then divide all weights by their g.c.d, if it is greater than $1$.)
Now an easy calculation confirms that we, indeed, constructed
(an uncountable family of) geodesic multinets with $3$ unbalanced
vertices and $3k$ balanced vertices, where $k$ can be arbitrarily
large. 

However, we would like to make the following observations:
\begin{enumerate}
\item The weights of at least some of the edges (e.g. $A_i^kA_i$) become unbounded, as $k\longrightarrow\infty$.
\item In fact, the total imbalance will increase linearly with $k$, as $k\longrightarrow\infty$.
\item The condition of rationality of the trigonometric functions
of $\frac{\alpha_i}{2}$ is very restrictive. We were able to 
carry out our construction only for a set of triples of
points $A_1$, $A_2$, $A_3$ of measure $0$ in the space of all
vertices of triangles in the Euclidean plane.
\end{enumerate}
\end{example}

Looking at this example, one might be led into thinking that the constructed
geodesic multinets with $3$ unbalanced vertices and arbitrarily
many balanced vertices can be converted into a geodesic net by some
sort of a small perturbation, where the balanced vertices
are replaced by ``clouds" of nearby points (with some extra edges inside each cloud), and all multiple
edges are replaced by close but distinct edges running between
chosen nearby ``copies" of their former endpoints. It is
easy to believe that such a perturbation plus, maybe, some
auxiliary construction will be sufficient to construct
examples of geodesic nets in the plane with $3$ unbalanced
vertices and an arbitrary number of balanced vertices.
Yet all such hopes are shattered by the following theorem
of one of the authors (F.P.):

\begin{theorem}[\cite{Pthree}]\label{thm:parschthree}
A geodesic net with $3$ unbalanced vertices $A_1$, $A_2$, $A_3$ in the Euclidean
plane has exactly one balanced vertex $O$ at the Fermat point of
the three unbalanced vertices and three edges $OA_i$.
Moreover, this assertion is true for geodesic
nets with $3$ unbalanced vertices on any non-positively 
curved Riemannian $\mathbb{R}^2$.
\end{theorem}

Note that \cite{Pthree} contains an example demonstrating that this
assertion is no longer true without the sign restriction on the curvature of the Riemannian plane. Yet it is not known if 
the assertion is still true if the integral of the positive
part of the curvature is sufficiently small. (The example for positive curvature constructed in \cite{Pthree} requires total curvature at least $\pi$.)

The striking contrast between Example \ref{ex:weightedthree} of geodesic multinets
with $3$ unbalanced vertices and the extreme rigidity of geodesic nets with three unbalanced
vertices on the Euclidean plane leads to some intriguing
open questions such as:

\begin{problem}
Let $\Sigma$ be the set of all triples $S$ of points
of the Euclidean plane such that there exist geodesic multinets
with $3$ unbalanced vertices at $S$ and arbitrarily many balanced
vertices. Is it true that $\Sigma$ is a set of measure zero (in $(\mathbb{R}^2)^3$)?
\end{problem}

\begin{problem} Is there a function $f(n)$ such that
for each geodesic multinet with three boundary vertices
in the Euclidean plane such the that multiplicities of all edges do not exceed $n$ the number of balanced vertices does not exceed $f(n)$?
\end{problem}

\begin{problem}
Classify all geodesic multinets in the
Euclidean plane with $3$ unbalanced vertices.
\end{problem}

\subsection{Geodesic nets in the Euclidean plane and $3$-space
with $4$ unbalanced vertices.}

We are going to start from the following remark. Given several
points $A_1,\ldots ,A_k$ in the Euclidean space $\mathbb{R}^n$,
there is always the unique point $O\in\mathbb{R}^n$ (called Fermat point) such that
the sum of distances $\sum_{i=1}^k \dist(x,A_i)$ attains its
global minimum at $O$. (This fact is an immediate corollary
of the convexity of the function $\sum_i \dist(A_i, x).$)
For three points forming a triangle with the angles $<120\degrees$, $O$ is the point such that all angles $A_iOA_j$ are equal to $120\degrees$. For four points
in the plane at the vertices of a convex quadrilateral,
$O$ is the point of intersection of the two diagonals. For four
vertices of a regular tetrahedron, $O$ is its center.
If $O$ is not one of the points $A_i$, then the star-shaped tree 
formed by all edges $OA_i$ is a geodesic net with unbalanced
vertices $A_1,\ldots , A_k$ and the only balanced vertex at $O$.
If $A_1, A_2, A_3, A_4$ are, say, vertices of a square or
a rectangle close to a square one has two other well-known
and ``obvious" geodesic nets with unbalanced vertices at $A_i$;
Both these nets are \includegraphics{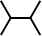} shaped (see Figure \ref{fig:threeandfournet}): 
They have two new balanced
vertices $O_1$ and $O_2$ connected by an edge. Each balanced
vertex is connected by edges with a pair of unbalanced vertices,
so that all three angles at either $O_1$ or $O_2$ are $120$
degrees, and each of the four unbalanced vertices is connected
with exactly one balanced vertex. There are three ways
to partition a set of four vertices into two unordered pairs,
yet only those where the unbalanced vertices in each  pair
are connected by a side of the convex quadrilateral can ``work". Of course, the locations of balanced points $O_1, O_2$ will be different for the
two ways to partition the set of four sides of the quadrilateral into pairs.
(Note that exactly the same idea works for the regular tetrahedron: Each of three pairs of opposite edges gives rise
to a \includegraphics{tinytree} shaped geodesic net with two balanced vertices.)

\begin{figure}[b]
	\centering
	\includegraphics{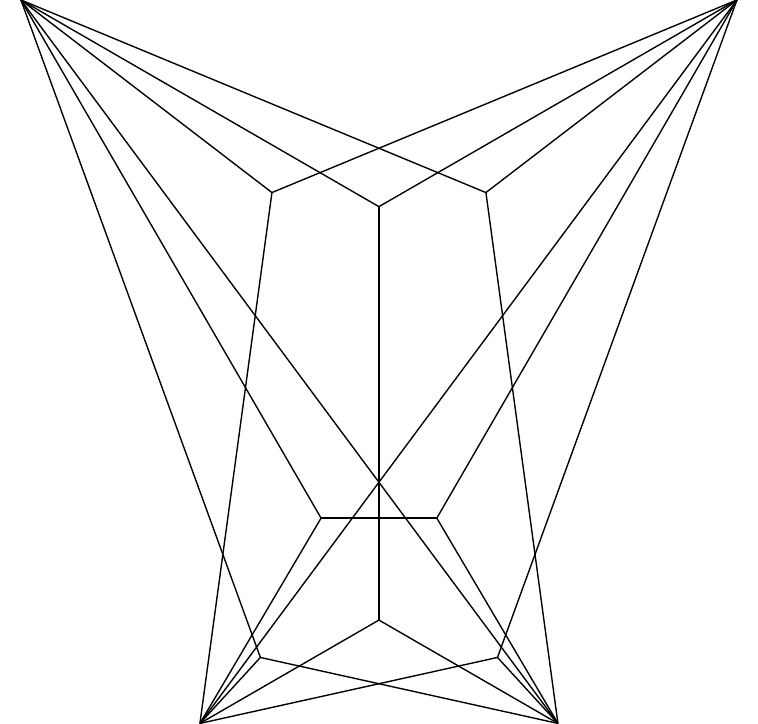}
	\caption{An example of a geodesic net in the plane with four unbalanced vertices which is an \enquote{overlay} of trees.}
	\label{fig:fournet}
\end{figure}

It had been observed in \cite{Pthree} that given vertices $A_1,\ldots ,A_4$ of
a convex quadrilateral close to a square but in general position,
one can combine the star-shaped net with one balanced point
at the point of intersection of diagonals, the two \includegraphics{tinytree} shaped
nets and four star-shaped geodesic nets with unbalanced points
at the vertices of each of $4$ triangles formed by all triples
of four vertices $A_i$ one obtains a geodesic net with $28$
balanced vertices (see Figure \ref{fig:fournet}). (One obtains some extra balanced vertices
as points of intersection of edges of geodesic nets that are
being combined.) This example might seem like a strong indication
that no analog of Theorem \ref{thm:parschthree} for geodesic nets with $4$ unbalanced vertices is possible. Yet one can define {\it irreducible}
geodesic nets on a given set $S$ of unbalanced vertices
as geodesic nets such that no subgraph formed by a proper subset
of the set of edges (with all incident vertices) is a geodesic
net with the same set $S$ of unbalanced vertices. It is clear
that classification of geodesic nets boils down to the classification of irreducible nets. As so far we have only
two ``obvious" isomorphism types of geodesic nets with $4$ unbalanced vertices (namely, X-shaped and \includegraphics{tinytree} shaped trees),
one might still suspect that there exists an easy classification
of geodesic nets with $4$-vertices on the Euclidean plane.

Yet the situation changed (at least for us) after one of the authors discovered a new example of an irreducible geodesic net
with $4$ unbalanced vertices at four vertices of the square
and $16$ balanced vertices (see Figure \ref{fig:fourirreducible}) A detailed
description of this example can be found in \cite{Pirreducible}. Now a natural next step in classification of geodesic nets
on $4$ vertices in the plane will be the following problem:

\begin{problem}\label{problem:fourmorethan16}
Find an irreducible geodesic net with
$4$ unbalanced vertices in the Euclidean plane with more than
$16$ balanced vertices (or prove that such a geodesic net does not
exist.)
\end{problem}

In fact, we believe that:

\begin{conjecture}\label{conj:fourandarbitrary}
There exist geodesic nets in the Euclidean plane with $4$ unbalanced vertices and an arbitrarily
large number of balanced vertices.
(Moreover, we will not be surprised if this assertion is already true in the case
when the set of unbalanced vertices coincides with the set of vertices of a square).
\end{conjecture}

\begin{figure}
    \centering
    \includegraphics{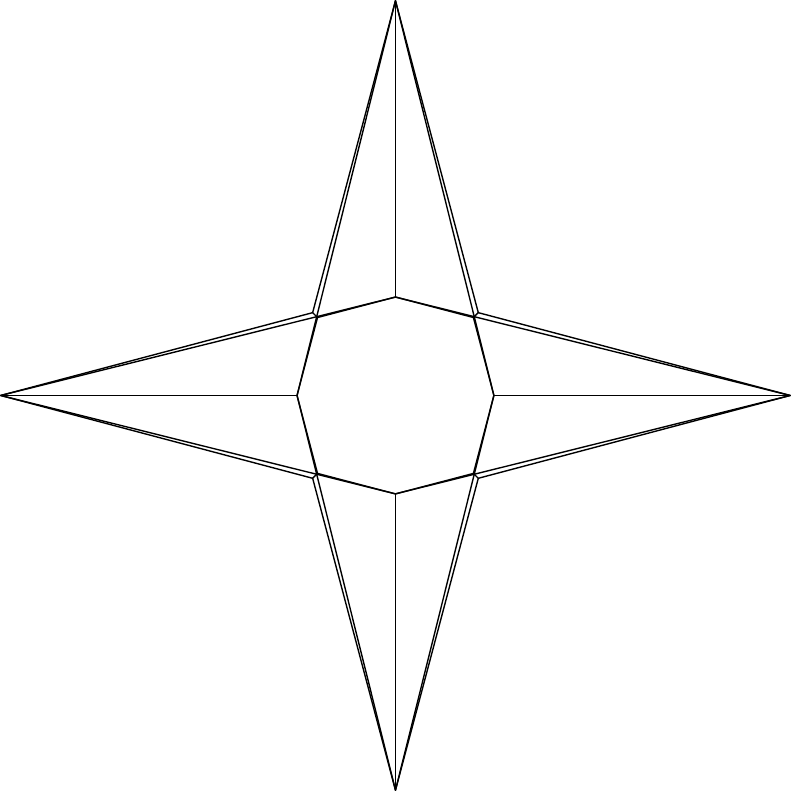}
    \caption{The example of an irreducible geodesic net with four unbalanced vertices that is not a tree, as constructed in \cite{Pirreducible}}
    \label{fig:fourirreducible}
\end{figure}

Note that we are not aware of any analogs of this geodesic net with $4$ unbalanced and $16$ balanced vertices when $4$ unbalanced vertices are non co-planar points in the Euclidean
$3$-space, e.g. the vertices of the regular tetrahedron.
Yet in this case there exists (a more obvious) geodesic
net with $4$ unbalanced vertices $A_i$ and $7$ balanced
vertices obtained as follows (see figure \ref{fig:four3d}): Start from the star-shaped geodesic net with the balanced vertex at the center of the regular tetrahedron. For each of $6$ triangles $A_iA_jO$, where
$i,j$ run of the set of all unordered distinct pairs of numbers
$1,2,3,4$ attach the Y-shaped geodesic net with unbalanced
vertices at $A_i$, $A_j$ and $O$ and a new balanced vertex at
the center of the triangle $A_iA_jO$. Nevertheless, it seems that
it is harder to construct irreducible nets with unbalanced
vertices at the vertices of a regular tetrahedron than at the vertices of a square. We would not be surprised if the answer
for the following question turns out to be positive:

\begin{problem}Is there a number $N$ such that each
irreducible geodesic net with unbalanced vertices at all vertices
of a regular tetrahedron has at most $N$ balanced vertices?
\end{problem}

\begin{figure}
    \centering
    \includegraphics{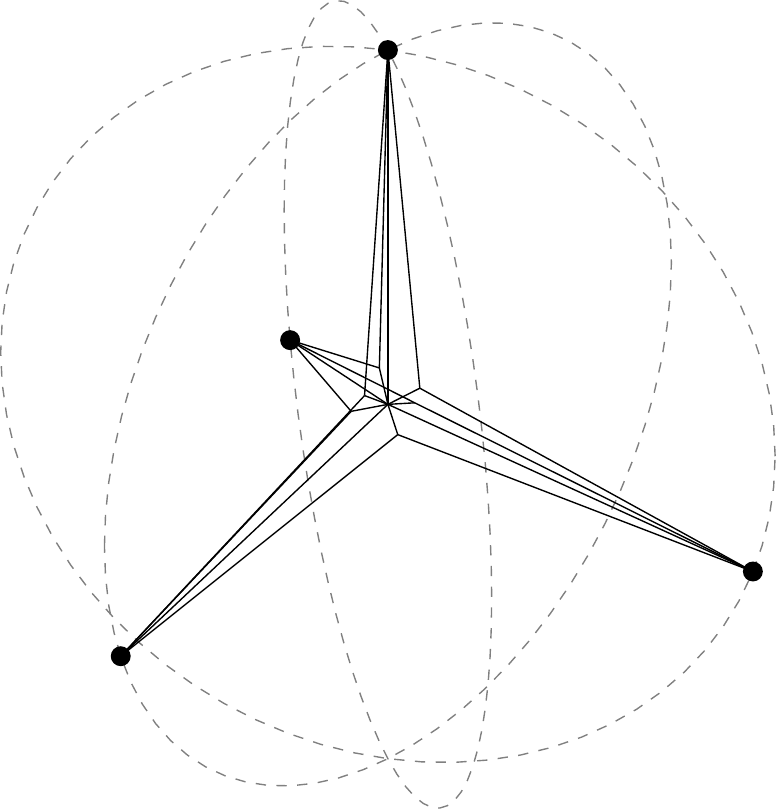}
    \caption{A geodesic net in Euclidean three-space, with four unbalanced and seven balanced vertices}
    \label{fig:four3d}
\end{figure}

\subsection{Geodesic nets in the plane: can one bound the number
of balanced vertices in terms of the number of unbalanced vertices?}\label{sect:canonebound}

We cannot solve Problem \ref{problem:fourmorethan16}. Yet in the next section, we are
going to describe a construction of a certain family of irreducible geodesic multinets $G_i(\varphi)$ with $14$ unbalanced vertices ($7$ of which a constant
and $7$ variable) and arbitrarily many balanced vertices. We
{\it believe} that these geodesic multinets are, in fact,
geodesic nets. Our faith is based on the following facts:

\begin{enumerate}
\item We checked numerically that the first $100$ geodesic multinets
from our list are, indeed, geodesic nets. (The number of balanced
vertices of $G_i(\varphi)$ is greater or equal than
$7i$.)
\item We constructed an sequence of functions $\varphi_i$ of one real variable $\varphi$. If for each $N$, some $N$ functions $\varphi_i(\varphi)$ are pairwise
distinct in a neighbourhood of $0$, then our construction,
indeed, produces geodesic nets with at least $7N$ balanced
vertices. The functions $\varphi_i(\varphi)$ are presented by a very complicated set of recurrent relations. Whenever there seem to be no reason
for any pair of these functions to coincide, the formulae are so
complicated that the proof of this fact eludes us.
\end{enumerate}

Note, that while the imbalances at the seven constant vertices
are unbounded, the imbalance at $7$ variable vertices remain
bounded. This leads us to a belief that some modification
of our construction might lead to elimination of several
variable unbalanced points leaving us only with seven constant
unbalanced points. Moreover, we believe that it is possible that our construction
will ``survive" small perturbations of the seven constant
unbalanced points. As a result we find that the following
conjecture
is very plausible:

\begin{conjecture} 1. There exist $N_0$ and an $N_0$-tuple $S$ such that for each $N$ there
exists a geodesic net with $S$ being its
set of unbalanced vertices
and the number of balanced vertices  greater than $N$.
2. Furthermore, there exist not merely
one such $N_0$-tuple $S$ but a subset of $(\mathbb{R}^2)^{N_0}$ of
positive measure (or even a non-empty
open subset) of such $N_0$-tuples.
\end{conjecture}

In fact, it is quite possible that $N_0=4$.

\subsection{Gromov's conjecture.} As we saw, even for the
simplest geodesic multinets in the Euclidean plane, there is no upper bound for
the number of balanced vertices in terms of the number
of unbalanced vertices. The example mentioned in section
\ref{sect:canonebound} and explained in detail in the next section strongly
suggests, that such a bound does not exist already for geodesic nets. The length of a geodesic net cannot be
of great help either, as we can rescale any geodesic net
to an arbitrarily small (or large) length without changing
its shape. One appealing conjecture due to M. Gromov is
the following:

\begin{conjecture}[M. Gromov]\label{conj:gromov}
The number of balanced
vertices of a geodesic net in the Euclidean plane can be
bounded above in terms of the number of unbalanced vertices
and the total imbalance.
\end{conjecture}

\begin{figure}
	\centering
	\includegraphics{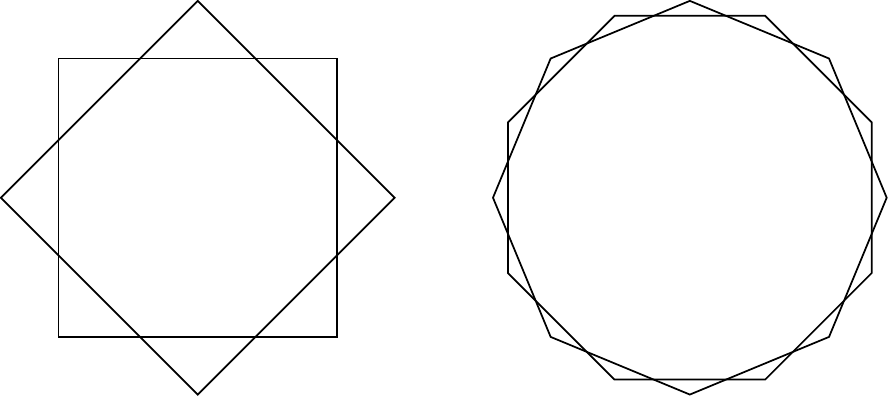}
	\caption{This construction shows that there is a sequence of nets with bounded imbalance, but arbitrarily many balanced vertices}
	\label{fig:onlyimbalancenotenough}
\end{figure}

In fact, we do not see any reasons why this conjecture
cannot be extended to geodesic multinets. Note that
the following simple example demonstrates that one cannot
majorize the number of balanced points only in terms of
the total imbalance without using the number of unbalanced
vertices (see figure \ref{fig:onlyimbalancenotenough}): Take a copy of a regular $N$-gon, and obtain
a second copy by rotating it by $\frac{\pi}{2N}$ about
its center. Take a geodesic net obtained as the union
of these two copies of the regular $N$-gon. The set of
unbalanced vertices will consist of $2N$ vertices of
both copies. Yet as sides of two copies intersect, we are going to obtain also $2N$ balanced vertices that arise
as points of intersections of various pairs of sides.
The imbalance at each unbalanced vertex is $2\sin\frac{\pi}{N}$, so the total imbalance is
$4N\sin\frac{\pi}{N}<4N\frac{\pi}{N}=4\pi$. We see
that when $N\longrightarrow\infty$, the number of balanced
points also tends to $\infty$, yet the total imbalance 
remains uniformly bounded.

The above conjecture by Gromov was published in the paper by 
Y. Mermarian \cite{memarian} for geodesic nets such that all imbalances
are equal to one (in our terms. As in this case the total
imbalance is equal to the number of unbalanced vertices,
the conjecture is that the number of unbalanced vertices does not exceed the value of some function of the number of balanced vertices. Note also that \cite{memarian} contains the proof of this restricted
version of the conjecture in cases, when the degrees of all balanced vertices are either all equal to $3$, or all are equal to $4$.) Yet, the following simple observation implies that the restricted form (imbalances equal to 1 at each unbalanced vertex) is, in fact, equivalent to full Conjecture \ref{conj:gromov}.
The observation is that if $v$ is an unbalanced vertex,
then it can be extended by adding less than $\imb(v)+3$
new edges starting at $v$ so that $v$ becomes balanced.
Applying this trick to all imbalanced vertices we replace
our original geodesic net by a new one, with the new number
of unbalanced vertices not exceeding the sum of the total imbalance and thrice the number of unbalanced vertices in the original net.
In this new geodesic net the
imbalances of all unbalanced vertices are equal to one.
Thus, the restricted version of the conjecture implies
the general version. We are going to explain this
observation  in the case, when $\imb(v)\in (0,1)$ leaving
the general case to the reader. In this case we need
to find three new edges starting at $v$ such that their
angles with the imbalance vector $\Imb(v)$ that we denote
$\alpha_1$, $\alpha_2$ and $\alpha_3$ satisfy the balancing condition that can be written in the scalar form as the system of two equations: $\sum_{i=1}^3\cos \alpha_i=2\imb(v)$ and $\sum_{i=1}^3\sin\alpha_i=0.$
It is clear that this system has an uncountable set of solutions. This fact enables us to ensure that none of the
new edges coincide with already existing edges incident to $v$.

Note that formula (1.2.2) implies that the length
of a geodesic net does not exceed the product of its
total imbalance and the diameter ( which for geodesic
nets in the Euclidean space is always
equal to the maximal distance between two unbalanced
points). Further, we can always rescale a geodesic net in the plane so that its diameter becomes equal to $1$. In this case its length becomes equal to $\frac{L}{D}$, where 
$L$ and $D$ are the values of the length and the diameter before the rescaling.
Therefore, Conjecture \ref{conj:gromov} would follow
from the validity of the following conjecture:

\begin{conjecture}
There exists a function $f$ of
such that each geodesic multinet in 
the Euclidean plane with $n$ unbalanced vertices, diameter $D$, and total length $L$ has less than $f(\frac{L}{D},n)$ balanced vertices.
\end{conjecture}

Now we would like to combine this conjecture with the
Becker-Kahn problem \ref{problem:beckerkahn} and extend it
to all Riemannian manifolds. Before doing so, consider
the example of a complete non-compact Riemannian manifold which is a disjoint union of (smooth) capped cylinders that have a fixed length but are getting thinner. More specifically, the cylinders have radii $\frac{1}{n}$ for all positive integers $n$ but fixed diameter of $1$. On any of these cylinders, we can now add $N$ closed geodesics around the waist of the cylinder, connecting all of them with a single closed geodesic that travels twice along the diameter of the manifold. Such a net will have $2N$ balanced vertices and fixed diameter $D=1$, but as long as $N/n$ is small enough, the length of the net $L$ gets arbitrarily close to $2$. So both $L/D$ and $L$ stay bounded whereas the number of balanced vertices can be chosen to be arbitrarily large. Note that we could make this manifold connected by connecting consecutive cylinders by thinner and thinner tubes of length $1$.

This example shows that for general Riemannian manifolds, we can't bound the number of unbalanced vertices in terms of $L/D$, $L$, and the number of balanced vertices. So we must either bound the injectivity radius of our Riemannian manifold $M$ from
below, or, more generally, adjust the length as follows:
The {\it adjusted total length} $\tilde L$ of a geodesic net is the sum of
integrals over all edges $e_i$ parametrized by their respective arclengths of $\frac{1}{\inj(e_i(s))}$, where $\inj(e_i(s))$ denotes the injectivity radius of the ambient
Riemannian manifold at $e_i(s)$. If $M$ is a Riemannian manifold
with a positive injectivity radius $\inj$, then
$\tilde L\leq \frac{L}{\inj}$. Now we can state our
most general conjecture.

\begin{conjecture}[Boundedness conjecture for geodesic nets on Riemannian manifolds] Let $M$ be a complete Riemannian manifold. There exists a function $f_M$ which depends
on $M$ but is invariant with respect to rescalings of $M$ with the following property: Let $G$ be a geodesic
net on $M$ with total length $L$, adjusted length $\tilde L$ and diameter $D$ that has $n$ unbalanced vertices.
Then its number of balanced vertices does not exceed
$f_M(\tilde L,\frac{L}{D},n)$. In particular, if
$M$ has injectivity radius $\inj>0$, then the number
of balanced vertices does not exceed
$f_M(\frac{L}{\inj}, \frac{L}{D},n).$
\end{conjecture}

\pagebreak
\section{The Star}
\label{sect:thestar}

%%%%%%%%%%%%%%%%%%%%%%%%%%%%%%%%%%%%%
\subsection{Overview}
%%%%%%%%%%%%%%%%%%%%%%%%%%%%%%%%%%%%%

The \enquote{star} $G_n(\varphi)$ constructed in this section is a possible example of how a geodesic net can be constructed that fulfills the following requirements:
\begin{itemize*}
	\item The net has 14 unbalanced vertices (of arbitrary degree)
	\item The net has an arbitrarily large (finite) number of balanced vertices
	\item All edges have weight one (as is required by our definition of geodesic nets)
\end{itemize*}
In fact, the third condition is what makes the present construction both interesting but also quite sophisticated. If we allowed integer weights on our edges, there would be much simpler constructions of geodesic nets with just three unbalanced vertices and an arbitrary number of balanced vertices (see \ref{fig:weightedthree}).

Our construction will work as follows: First, we will construct a highly symmetric geodesic net $G_n(0)$, layer by layer, that provides for an arbitrarily large number of balanced vertices. To arrive at a result as depicted in figure \ref{fig:star100}, we first need to build a toolbox to be used during the construction.

This highly symmetric net has edges of integer weights. That is why we will make sure that our construction works for a small deviation from the symmetric case as well, arriving at a net $G_n(\varphi)$. This deviation is intended to remove any integer weights.

As it turns out, showing that for some nonzero deviation $\varphi\in(-\epsilon,\epsilon)$, none of the edges of $G_n(\varphi)$ \enquote{overlap} necessitates a close look at a quite complicated finite recursive sequence. More precisely, we need to ensure that this sequence never repeats. We will present explicit formulas for this sequence as well as numerical results strongly suggesting that this sequence does in fact never repeat.

Assuming that this sequence never repeats, the \enquote{star} constructed in this section would therefore be an example for a sequence of geodesic nets with a fixed number of unbalanced vertices but an arbitrarily large number of balanced vertices.

%%%%%%%%%%%%%%%%%%%%%%%%%%%%%%%%%%%%%
\subsection{Construction Toolbox}
%%%%%%%%%%%%%%%%%%%%%%%%%%%%%%%%%%%%%

We will first build our \enquote{toolbox} to facilitate the construction of the geodesic net below.

\subsubsection{Suspending}

Suspending is a process that adds an additional edge to a vertex $v$ to change its imbalance.

\vspace{2mm}
\begin{mdframed}[style=greybox]
\begin{minipage}{0.8\textwidth}
\begin{method}[Single-hook suspension]
	Consider a vertex $v$ and another vertex $P$, called the \emph{hook}. We \emph{suspend} $v$ from $P$ by adding the edge $vP$.
\end{method}
\end{minipage}\hfill
\begin{minipage}{0.15\textwidth}
	\raggedleft
	\includegraphics{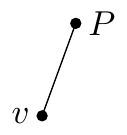}
\end{minipage}
\end{mdframed}
\vspace{3mm}
\begin{mdframed}[style=greybox]
    \setlength{\parindent}{0em}
	\begin{method}[Two-hook suspension]
	Consider a vertex $v$ and two other vertices $P$, $Q$ such that all three interior angles of the triangle $\Delta PvQ$ are less than $120\degrees$.
		
	\begin{minipage}{0.6\textwidth}
	\setlength{\parskip}{0.5em}
	\setlength{\parindent}{0em}
	There is a unique point $F$ -- called the Fermat point -- inside the triangle $\Delta PvQ$ such that the edges $PF$, $QF$ and $vF$ form angles of $120\degrees$ at $F$. It can be constructed as follows:
	\begin{itemize*}
		\item Let $X$ be the third vertex of the unique equilateral triangle that has base $PQ$ and that is lying outside the triangle $\Delta PvQ$.
		\item Let $c$ be the unique circle defined by the points $P$, $Q$ and $X$.
		\item Note that $c$ and the segment $Xv$ intersect at two points: $X$ itself and one other point. That other point is $F$.
	\end{itemize*}
	\end{minipage}\hfill
\begin{minipage}{0.35\textwidth}
	\raggedleft
	\includegraphics[width=\textwidth]{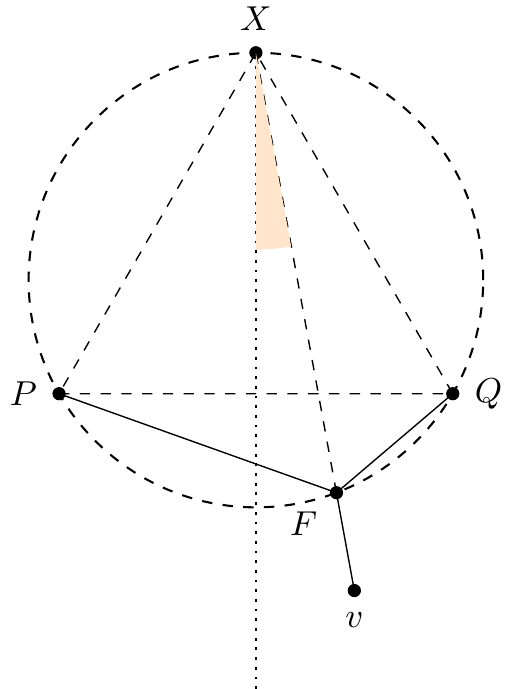}
	\end{minipage}
	That this construction does indeed yield the Fermat point (also known as the Toricelli point) is a result of classic Euclidean Geometry.
	
	We \emph{suspend} $v$ from $P$ and $Q$ by adding $F$ and the edges $PF$, $QF$ and $vF$. Note that now $F$ is a degree three balanced vertex.
	
	The orange angle between $Xv$ and the axis of symmetry of the equilateral triangle will be denoted by $\varphi$ later. Note that if $\varphi=0$, then the picture is symmetric under reflection along $vX$.
	\end{method}
\end{mdframed}

\subsubsection{Winging}

Winging is a process that turns an unbalanced vertex into a balanced vertex.

\vspace{2mm}
\begin{mdframed}[style=greybox]
\begin{minipage}{0.7\textwidth}
\setlength{\parskip}{0.5em}
\setlength{\parindent}{0em}
\begin{method}[Winging a degree $2$ vertex]\label{meth:wing2}
	Consider an unbalanced vertex $v$ of degree $2$ with $\alpha_i$ being the larger angle between the two incident edges, i.e. with $180\degrees<\alpha_i<360\degrees$. We can balance this vertex by \enquote{spreading wings} as follows: Extend the two incident edges to the other side of the vertex, resulting in a degree $4$ balanced vertex.
	
	If $\beta_i$ is the smaller of the two angles between the two new edges (\enquote{wings}), then $\beta_i=360\degrees-\alpha_i$.
\end{method}
\end{minipage}\hfill
\begin{minipage}{0.25\textwidth}
	\raggedleft
	\includegraphics[width=\textwidth]{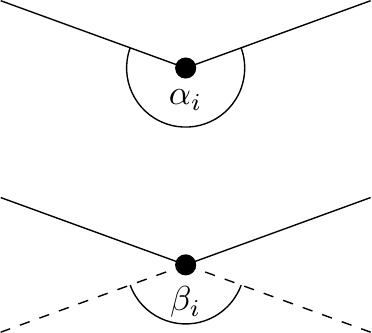}
\end{minipage}
\end{mdframed}
\vspace{3mm}
\begin{mdframed}[style=greybox]
\begin{minipage}{0.7\textwidth}
\setlength{\parskip}{0.5em}
\begin{method}[Winging a degree $3$ vertex]\label{meth:wing3}
	Consider an unbalanced vertex $v$ of degree $3$ such that the total imbalance (i.e. the sum of the unit vectors parallel to an edge) is less than $2$.
	
	We can balance this vertex by adding two edges in a unique way as follows: Since the imbalance is a vector of length less than $2$, there is one (and only one) way of writing its inverse as the sum of two unit vectors. Add the two corresponding edges that balance the vertex in this way (these edges might coincide with existing edges). We arrive at a balanced vertex of degree $5$.
\end{method}
\end{minipage}\hfill
\begin{minipage}{0.25\textwidth}
	\raggedleft
	\includegraphics[width=\textwidth]{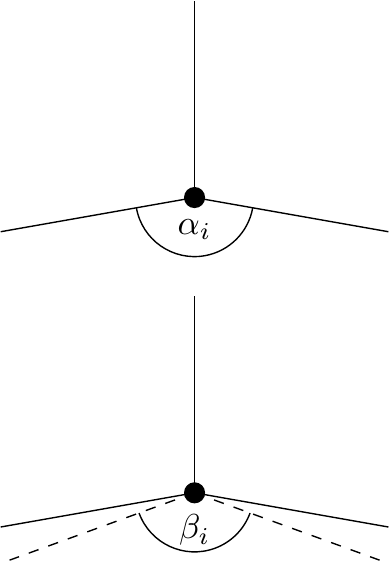}
\end{minipage}

	Note that this construction does \emph{not} require the picture to be symmetric as in the sketch on the right. However, in the case that it is in fact symmetric, it is important to point out a special relationship: After winging, the picture will remain symmetric and we also get the following angle relation: If we denote the smaller of the two angles between the newly added wings by $\beta_i$, basic trigonometry yields $\beta_i=2\cdot\arccos(1/2-\cos(\alpha_i/2))$. Also, as long as $\alpha_i\neq 2\arccos(1/4)$ the two dashed edges will \emph{not} coincide with already present edges since then $\beta_i\neq\alpha_i$.
\end{mdframed}

\subsubsection{About algebraic angles}

Recall the following theorem based on Lindemann-Weierstra\ss:

\begin{theorem}
	If the angle $\alpha$ is algebraic (in radians), then $\cos(\alpha)$ and $\sin(\alpha)$ are transcendental.
\end{theorem}

We will fix an angle $\alpha_0>240\degrees$ which will be close to $240\degrees$, but so that $\alpha_0$ is in fact algebraic (and therefore its sine and cosine are transcendental, a property that we will need below). We will choose $\alpha_0=88/21 \textrm{ (rad)}\approx 240.1\degrees$, but of course any other algebraic angle closer to $240\degrees$ would also work.

\subsubsection{The parameters $n$ and $\varphi$}

The construction of the geodesic net $G_n(\varphi)$ relies on two parameters $\varphi$ and $n$.

We will start with an \emph{outer circle}, that is fixed and doesn't change under any of the parameters. We the proceed and construct an \emph{inner circle} whose deviation from the symmetric case is measured by the angle $\varphi$. This inner circle is the \enquote{zeroth layer} of the construction. We will then add a total of $n$ layers, producing more and more balanced vertices while keeping the number of unbalanced vertices fixed.

\subsubsection{Outer circle}

The \emph{outer circle} is given by seven equiangularly distributed points on a circle. These seven vertices will be one half of the $14$ unbalanced vertices of the resulting net.

Note that the whole construction will be scaling invariant, so we can choose an arbitrary radius for the outer circle. We will fix the scale of the picture further below.

Whenever we will use the process of \emph{suspending} a vertex as defined above, the two \textit{hooks} will be two neighbouring vertices on the outer circle.

\begin{figure*}[!htb]
	\centering
	\includegraphics{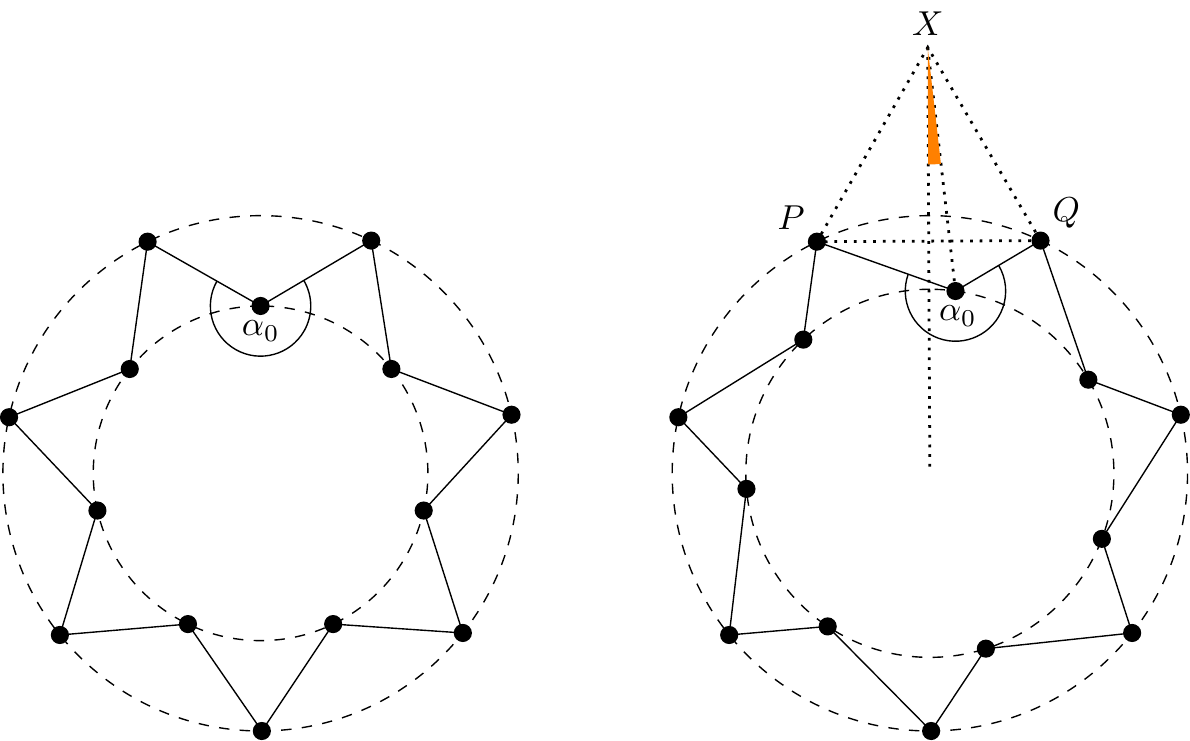}
	\caption{Outer circle and inner circle (= zero-th layer), symmetric case (left) and deviated case (right). Note that for fixed $\alpha_0$, there is a unique way to construct the inner circle from the outer circle once the deviation $\varphi\in(-\epsilon,\epsilon)$ is given. The reference point for deviation is given by $X$, which is the third point of the equilateral triangle that has two adjacent points $P$ and $Q$ on the outer circle as base. Note that the seven points on the inner circle are always deviated by the same angle $\varphi$. In other words: with deviation, there is no reflectional symmetry anymore, but rotational symmetry is maintained.}
	\label{fig:circles}
\end{figure*}

\subsubsection{Inner circle}

The \emph{inner circle} is defined as follows: First, we fix $\alpha_0$ as specified above. Note that this angle will not change under deviation later. Fix two neighbouring points $P$ and $Q$ on the outer circle and let $X$ be the third vertex of the equilateral triangle with base $PQ$ that lies outside the outer circle (see figure \ref{fig:circles}). Recall that we are provided a deviation angle $\varphi\in(-\epsilon,\epsilon)$. Consider the segment $OX$ (where $O$ is the canter of the outer circle) and rotate this segment around $X$ by $\varphi$ There is a unique vertex $v$ on this segment such that $\angle PvQ=\alpha_0$. This is one vertex of the inner circle. The other six vertices of the inner circle are then provided by rotational symmetry (again, see figure \ref{fig:circles}).

We then connect the outer and inner circles by edges as depicted. To later simplify calculations, we now scale the picture so that the radius of the inner circle is 1.

%%%%%%%%%%%%%%%%%%%%%%%%%%%%%%%%%%%%%
\subsection{The construction}
%%%%%%%%%%%%%%%%%%%%%%%%%%%%%%%%%%%%%

\subsubsection{Overview of the construction}

The initial setup of \emph{outer and inner circle} as described above is denoted as $G_0(\varphi)$. It is trivially a geodesic net with 14 unbalanced and no balanced vertices.

We will now add \emph{layers} to $G_0(\varphi)$ to arrive at geodesic nets $G_0(\varphi),G_1(\varphi)\dots,G_n(\varphi)$ with the following properties:
\begin{itemize*}
	\item Each $G_i(\varphi)$ has $14$ unbalanced vertices.
	\item The number of balanced vertices goes to infinity as $i\longrightarrow\infty$.
\end{itemize*}
So by choosing $n$ large enough, we get a geodesic net $G_n(\varphi)$ with $14$ unbalanced vertices and $N$ balanced vertices.

There is, however, one important caveat: We want the geodesic net to only have edges of weight one. As stated above: If we allowed for weighted edges, much simpler examples could be constructed.

In light of that requirement, we will observe the following:

\begin{itemize*}
	\item If we do \emph{not} introduce deviation, i.e. if we fix $\varphi=0$, we get highly symmetric geodesic nets $G_i(0)$, many edges of which will intersect non-transversally. This means that some edges would need to be represented using integer weights.
	\item However, for small $\varphi\in(-\epsilon,\epsilon)$, we get geodesic nets $G_i(\varphi)$ with significantly less symmetry and for which numerical results strongly suggest that all edges intersect transversally (if at all). So they are in fact nets with edges of weight one.
\end{itemize*}

\subsubsection{About the feasibility of the process and smooth dependence}

As we will see in the constructive process below, there are certain requirements on the behaviour of angles and lengths that are necessary to make this construction possible. We will proceed as follows:
\begin{itemize*}
	\item We will first describe the construction, which will be the same for the non-deviated and the deviated case. This construction will use several such requirements.
	\item We will then prove that all requirements are fulfilled for the non-deviated case $\varphi=0$.
	\item We will observe that each iteration of the construction smoothly depends on the previous one.
	\item Since all our requirements turn out to be restrictions of angles and lengths to open intervals and since the construction smoothly ($\Rightarrow$ continuously) depends on the initial setup, the requirements will therefore also be fulfilled for small $\varphi\in(-\epsilon,\epsilon)$.
\end{itemize*}

\subsubsection{Iterative process}

\begin{figure*}
	\centering
	\includegraphics[width=\textwidth]{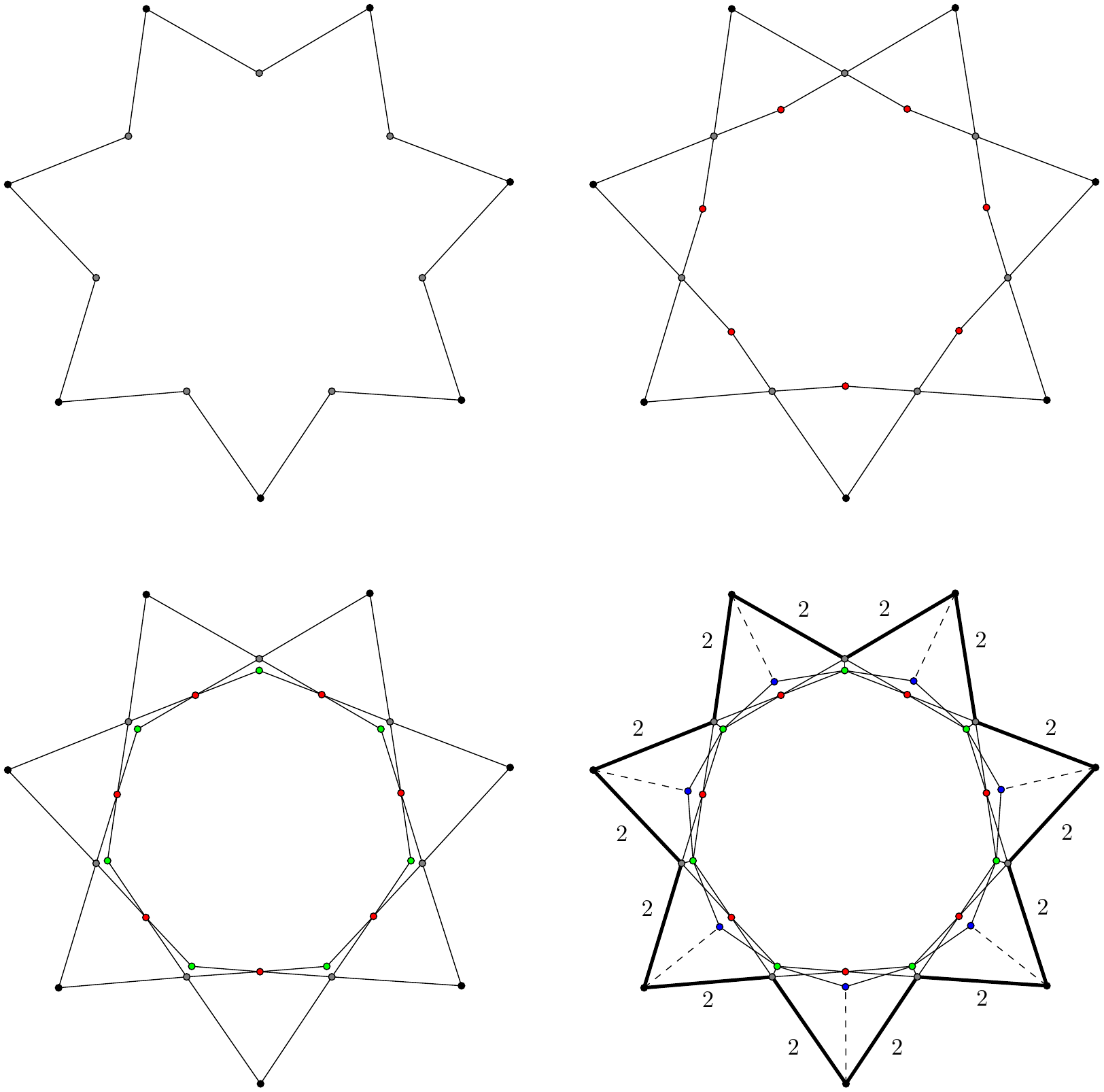}
	\caption{First steps of the construction for $\varphi=0$.\newline
	\textbf{Top left:} Outer vertices ($V_{-1}$, black) and inner vertices ($V_{0}$, grey)\newline
	\textbf{Top right:} The vertices of $V_0$ (grey) have been winged and the wings meet at the new vertices of $V_1$ (red)\newline
	\textbf{Bottom left:} The vertices of $V_1$ (red) have been winged and the wings meet at the new vertices of $V_2$ (green)\newline
	\textbf{Bottom right:} The vertices of $V_2$ (green) first were suspended (note the edges from green to grey and the double weight on the outer edges) and then were winged. The wings meet at the new vertices of $V_3$ (blue)\newline
	\textbf{After these first three steps}, the seven vertices on the outer circle as well as the vertices of $V_3$ are unbalanced. There are 21 balanced vertices indicated. During the construction, we also get additional \enquote{accidental} degree four balanced vertices at points of intersection.\newline
	\textbf{For the next step}, the dashed edges would be added to \emph{suspend} the vertices of $V_3$ (blue) and then each of them would be \enquote{winged} again.}
	\label{fig:iterations}
\end{figure*}

We denote the set of vertices on the outer circle by $V_{-1}$ and on the inner circle by $V_0$ and proceed to construct $V_i$ for $i\geq 1$.

The reader is encouraged to first consult figure \ref{fig:iterations} that explains the process visually.

Consider the vertices of $V_i$, each of which is a degree $2$ vertex that is adjacent to two vertices of $V_{i-1}$. Using the $14$ connecting edges, we get a $14$-gon whose vertices alternate between vertices of $V_{i-1}$ and vertices of $V_i$. For the interior angle $\alpha_i$ at the vertices $V_i$ (\emph{not} at the vertices of $V_{i-1}$), one of the following two cases can occur: $\alpha_i>180\degrees$, called Case A; or $\alpha_i<180\degrees$, called Case B (We justify $\alpha_i\neq 180\degrees$ for all $i$ later).

\textbf{Case A: $\alpha_i>180\degrees$.} In this case, the vertices of $V_i$ are unbalanced vertices of degree $2$ such that we can \textit{wing a degree $2$ vertex} getting an angle $\beta_i=360\degrees-\alpha_i<180\degrees$ as described above. Each wing will end as soon as it intersects with another wing. At those seven points of intersection, we fix the seven vertices of $V_{i+1}$. Proceed to the next iteration.

\textit{Examples for Case A in figure \ref{fig:iterations} are the first two steps, namely the grey and red vertices.}

\textbf{Case B: $\alpha_i<180\degrees$.} In this case, we will first add an outwards edge to each vertex of $V_i$ using suspension. We distinguish two cases by the parity of $i$.

\textbf{Case B1, $i$ is even:}. In this case, by construction, each vertex $v$ of $V_i$ is close to a radial line through the origin and at the half-angle between two outer vertices $P$ and $Q$ (for deviation $\varphi=0$, $v$ is in fact \emph{on} that radial line). Consider the triangle $\Delta PvQ$. It is clear that $\angle vPQ,\angle PQv\leq 90\degrees<120\degrees$. Furthermore note that $v$ is inside the \emph{inner circle} (we will prove this in lemma \ref{thm:xlessthanone}). Even if $v$ were \emph{on} the inner circle, we would have $\angle QvP<120\degrees$ (by the choice of $\alpha_0 > 240\degrees$). So since $v$ is inside the inner circle, we still have $\measuredangle QvP<120\degrees$. Consequently, we can do a \emph{two-hook suspension} of $v$ from the hooks $P$ and $Q$ as described above.

\textit{And example for Case B1 in figure \ref{fig:iterations} is the third step, namely the green vertices.}

\textbf{Case B2, $i$ is odd:}. In this case, by construction, each vertex $v$ of $V_i$ is close to a radial line through the origin and one of the vertices $P$ on the outer circle (again, for deviation $\varphi=0$, $v$ is \emph{on} that radial line). We will do a \emph{one-hook suspension} of $v$ from $P$ by adding their connecting edge.

\textit{And example for Case B2 in figure \ref{fig:iterations} is the fourth step, namely the blue vertices.}

\textbf{After applying case B1 or B2}, each vertex of $V_i$ now is a degree $3$ unbalanced vertex. We will prove below that it has imbalance of less than $2$. Therefore, we can \emph{wing} this vertex of degree $3$. Each wing will end as soon as it intersects with another wing. At those seven points of intersection, we fix the seven vertices of $V_{i+1}$. Proceed to the next iteration.

This describes the whole construction. An example of the non-deviated case ($\varphi=0$) can be found in figure \ref{fig:star100}. We are left to show that the claims that make this construction possible are actually true.

\begin{figure*}
	\centering
	\includegraphics{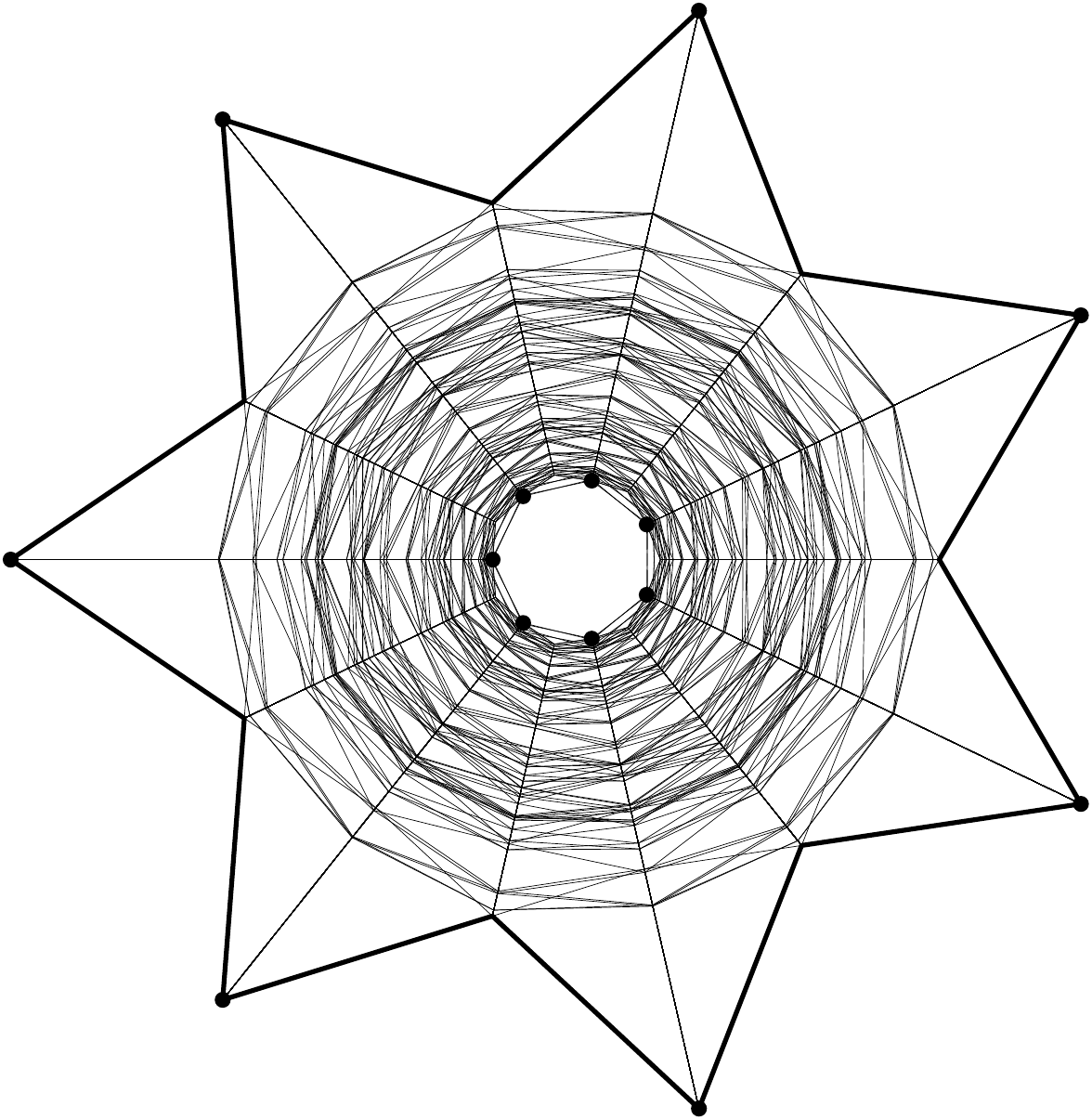}
	\caption{$G_{100}$ of the non-deviated case. The 14 circled vertices are the only unbalanced vertices of this geodesic net. Note that some edges have integer weights of more than $1$. Those are the radial edges, as well as the edges of the outermost $14$-gon. After introducing deviation, these edges will split into weight-one-edges.}
	\label{fig:star100}
\end{figure*}

\subsubsection{Helpful lemmata}

The above construction implicitly uses several geometric facts which we will prove in this section. It is interesting to note that parts of the following lemma could be proven similarly for a construction starting with more than $7$ outer vertices, but fail for $6$ vertices. More specifically, we prove $\alpha_i>120\degrees$ below, which would not be true if we started with $3$, $4$, $5$, $6$ vertices. This is the reason for the seemingly arbitrary choice of seven as the \enquote{magic number} of the construction.

Note the following:
\begin{lemma}
	The positions of all vertices depend smoothly on the deviation angle $\varphi\in(-\epsilon,\epsilon)$.
\end{lemma}

\begin{proof}
The outer circle never moves. The definition of the inner circle (which is the layer $V_0$) makes it clear that the position of the vertices of that layer depend smoothly on $\varphi$.

Since, to find all the other layers, we are using nothing but suspending and winging as defined previously, we only need to check those processes. Assume the position of the vertices up to $V_{i}$ depend smoothly on $\varphi$.
\begin{itemize*}
	\item The angles of the incoming edges to $V_i$ from $V_{i-1}$ only depend on the positions of $V_i$ and $V_{i-1}$ which depend smoothly on $\varphi$ by induction hypothesis.
	\item If \emph{one-hook suspension} is necessary: $P$ is on the outer circle, so it doesn't change under $\varphi$. Since the position of $v$ depends smoothly on $\varphi$, so does the angle of the hooking edge. So the imbalance of $v$ before winging will change smoothly.
	\item If \emph{two-hook suspension} is necessary: $P$, $Q$ and $X$ don't change under $\varphi$. Since the position of $v$ depends smoothly on $\varphi$, so does the angle of the hooking edge. So the imbalance of $v$ before winging will change smoothly.
\end{itemize*}
We now have established that the angles of all incoming edges to the vertices of $V_i$ and therefore the imbalance at the vertices of $V_i$ after possible suspension depends smoothly on $\varphi$. Checking the two possibilities for winging, it is apparent that the angles of the outgoing wings depend smoothly on the imbalance. Since the vertices of the next layer are defined to be the intersection of those wings, the positions of the next layer depend smoothly on $\varphi$.
\end{proof}

Note that all the following lemmas assert inequalities regarding angles and distances. In light of the previous lemma, it is therefore enough to prove them for $\varphi=0$. By smooth dependence (and therefore continuous dependence), they are then still true for small $\varphi\in(-\epsilon,\epsilon)$.

We will first prove the following technical lemma, the usefulness of which will be apparent later.

\begin{lemma}\label{thm:anglecalc}
	Consider the angles $\alpha_i$ and $\beta_i$ as the angle between the incoming edges and the angle between the outgoing edges during winging (see the figures describing the winging process). We have:
	\begin{enumerate*}
		\item $\alpha_i\neq 180\degrees,2\arccos(1/4)$ for all $i\geq 0$
		\item $120\degrees<\alpha_i< 190\degrees$ for all $i\geq 1$
		\item $120\degrees<\beta_i<180\degrees$ for all $i\geq 1$
	\end{enumerate*}
\end{lemma}

\begin{proof}
	As established, it is enough to consider the symmetric case $\varphi=0$.

	Recall that $\alpha_0=88/21 \textrm{(rad)}\approx 240.1$ and therefore $\beta_0\approx 119.9\degrees$. The formulas for $\beta_i$ depending on $\alpha_i$ were derived above when \emph{winging} was defined:
	\begin{align*}
		\beta_{i}&=\begin{cases}
			360\degrees-\alpha_i&\alpha_i>180\degrees\quad\text{(winging of degree 2 vertex)}\\
			2\cdot\arccos(1/2-\cos(\alpha_i/2))&\alpha_i<180\degrees\quad\text{(winging of degree 3 vertex)}\\
		\end{cases}
	\end{align*}
	Furthermore, since the vertices of $V_i$ and the vertices of $V_{i+1}$ form a 14-gon for which the interior angles at $V_i$ are $\beta_i$ (outgoing edges) and the interior angles at $V_{i+1}$ are $\alpha_{i+1}$ (incoming edges), we have
	\begin{align*}
		7\beta_i+7\alpha_{i+1}=12\cdot180\degrees\Leftrightarrow\alpha_{i+1}=\frac{12\cdot180\degrees}{7}-\beta_i
	\end{align*}
	For a visualization of the interdependence of these sequences see figure \ref{fig:alphabetaxdependence}.
	
	\begin{figure}
	\centering
	\includegraphics{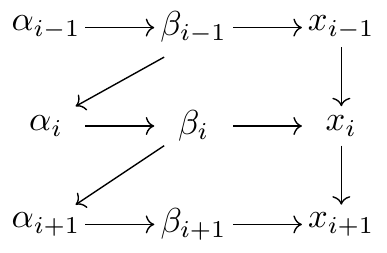}
	\caption{The interdependences of $\alpha_i$, $\beta_i$ and $x_i$}
	\label{fig:alphabetaxdependence}
\end{figure}
	
	Note that after proving (a), it is indeed clear that we do not need to consider the case $\alpha_i=180\degrees$. We proceed by induction. Note that (a) starts at $i=0$ whereas (b) and (c) start at $i=1$:
	\begin{enumerate*}
		\item Recall that $\cos(\alpha_0)$ is not algebraic (by our initial choice of $\alpha_0$, see above). We will prove the following fact which will imply the required result: $\cos(\alpha_i)$ is never algebraic for $i\geq 0$. The base case is given.
		
			To proceed with induction, first consider the case where $\alpha_i>180\degrees$ and therefore
			\begin{align*}
				&\alpha_{i+1}=12\cdot 180\degrees/7-360\degrees+\alpha_i\\
				\Rightarrow&\cos(\alpha_{i+1})=\cos(12\cdot 180\degrees/7-360\degrees+\alpha_i)\\
				\Rightarrow&\cos(\alpha_{i+1})=\cos(12\cdot 180\degrees/7+\alpha_i)\\
				\Rightarrow&\cos(\alpha_{i+1})=\cos(12\cdot 180\degrees/7)\cos(\alpha_i)-\sin(12\cdot 180\degrees/7)\sin(\alpha_i)\\
				\Rightarrow&\cos(\alpha_{i+1})=\cos(12\cdot 180\degrees/7)\cos(\alpha_i)-\sin(12\cdot 180\degrees/7)\sqrt{1-\cos^2(\alpha_i)}
			\end{align*}
			It follows that $\cos(\alpha_{i+1})$ is algebraic if and only if $\cos(\alpha_i)$ is algebraic.
			
			Similarly, if $\alpha_i<180\degrees$, then
			\begin{align*}
				&\alpha_{i+1}=12\cdot 180\degrees/7-2\cdot\arccos(1/2-\cos(\alpha_i/2))\\
				\Rightarrow&\cos(\alpha_{i+1}/2-6\cdot 180\degrees/7)=1/2-\cos(\alpha_i/2)\\
				\Rightarrow&\cos(\alpha_{i+1}/2)\cos(6\cdot 180\degrees/7)+\sin(\alpha_{i+1}/2)\sin(6\cdot 180\degrees/7)=1/2-\cos(\alpha_i/2)\\
				\Rightarrow&\cos(\alpha_{i+1}/2)\cos(6\cdot 180\degrees/7)+\sqrt{1-\cos^2(\alpha_{i+1}/2)}\sin(6\cdot 180\degrees/7)=1/2-\cos(\alpha_i/2)
			\end{align*}
			It follows that $\cos(\alpha_{i+1}/2)$ is algebraic if and only if $\cos(\alpha_i)$ is algebraic. But note that		
			\begin{align*}
				\cos(\alpha_{i+1})=2\cos^2(\alpha_{i+1}/2)-1
			\end{align*}
			Therefore $\cos(\alpha_{i+1})$ is algebraic if and only if $\cos(\alpha_{i+1}/2)$ is algebraic. The claim follows.
		\item The base case for $\alpha_1$ can be verified by calculation. Now assume $120\degrees<\alpha_i<190\degrees$ and consider $\alpha_{i+1}$. We have two cases:
	
			If $\alpha_i>180\degrees$, then
			\begin{align*}
				\alpha_{i+1}=\frac{2160\degrees}{7}-360\degrees+\alpha_i
			\end{align*}
			We can see that $120\degrees<\alpha_{i+1}<190\degrees$ is true provided $180\degrees<\alpha_i<190\degrees$.
			
			If $\alpha_i<180\degrees$, then
			\begin{align*}
				\alpha_{i+1}=\frac{2160\degrees}{7}-2\cdot\arccos(1/2-\cos(\alpha_i/2))
			\end{align*}
			It follows that $\alpha_{i+1}$ is an increasing function of $\alpha_i$ for $120\degrees<\alpha_i<180\degrees$ and that for $\alpha_i=120\degrees$ we get $\alpha_{i+1}\approx 129\degrees$ whereas for $\alpha_i=180\degrees$ we get $\alpha_{i+1}\approx 188\degrees$. So $120\degrees<\alpha_{i+1}<190\degrees$ is indeed the case.
		\item The bounds for $\beta_i$ are immediate from the bounds on $\alpha_i$ and the formula for $\beta_i$ above.
	\end{enumerate*}
\end{proof}

We can now proceed to prove facts that were relevant to our construction.

\begin{lemma}
	In the construction as defined above:
	\begin{enumerate*}
		\item The incoming edges never meet at an angle of $\alpha_i=180\degrees$, i.e. we always end up with case A or case B as described in the construction.
		\item The total imbalance before winging, even after possible suspension, is always less than $2$, i.e. winging is always possible.
		\item The outgoing edges produced by winging a degree 3 vertex never coincide with the incoming edges.
	\end{enumerate*}
\end{lemma}

\begin{proof}
	It is again enough to consider the symmetric case $\varphi=0$ since each of the asserted properties can be expressed as an inequality, so they remain true for small $\varphi\in(-\epsilon,\epsilon)$.
	\begin{enumerate*}
		\item This is given explicitly in the previous lemma.
		\item $\alpha_i$ is the angle between incoming edges. Since $120\degrees<\alpha_i<190\degrees$ (see previous lemma), the imbalance produced by the two incoming edges is always less than $1$. Suspension adds an imbalance of at most $1$. Therefore the total imbalance is less than $2$.
		\item We are considering the symmetric case. As stated in the definition of winging, the edges would only coincide if $\alpha_i=\beta_i$ which requires $\alpha_i=2\arccos(1/4)$, which is never the case by the previous lemma.
	\end{enumerate*}
\end{proof}

Finally, we asserted that adding a Fermat point for two-hook suspension is always possible if necessary. This assertion was based on the following fact:

\begin{lemma}\label{thm:xlessthanone} \emph{All layers $V_i$ lie strictly inside the inner circle for $i\geq 1$}. Furthermore, the radius of the layers goes to zero as $i\longrightarrow\infty$.
\end{lemma}

\begin{proof}
	We will again only consider the symmetric case. By smooth dependence on $\varphi$, the claim follows for the deviated case.
	
	Recall that we scaled the construction so that the inner circle is at a radius of $x_0=1$. During the construction as defined above, if we denote by $x_i$ the distance of the vertices at the $i$-th step from the origin, the claim follows if we prove $x_i<1$ for all $i\geq 1$ and that $x_i\longrightarrow 0$ as $i\longrightarrow\infty$. Note that in the symmetric case:
	\begin{align*}
		x_{i+1}&=x_if(\beta_i)&\text{where }f(\beta_i)&=\frac{\sin\beta_i/2}{\sin(1080\degrees/7-\beta_i/2)}&\text{(see figure \ref{fig:lawofsines})}
	\end{align*}
	where the formula for $\beta_i$, which in itself depends on $\alpha_i$, can be found above.
	
	\begin{figure}
	\centering
	\includegraphics{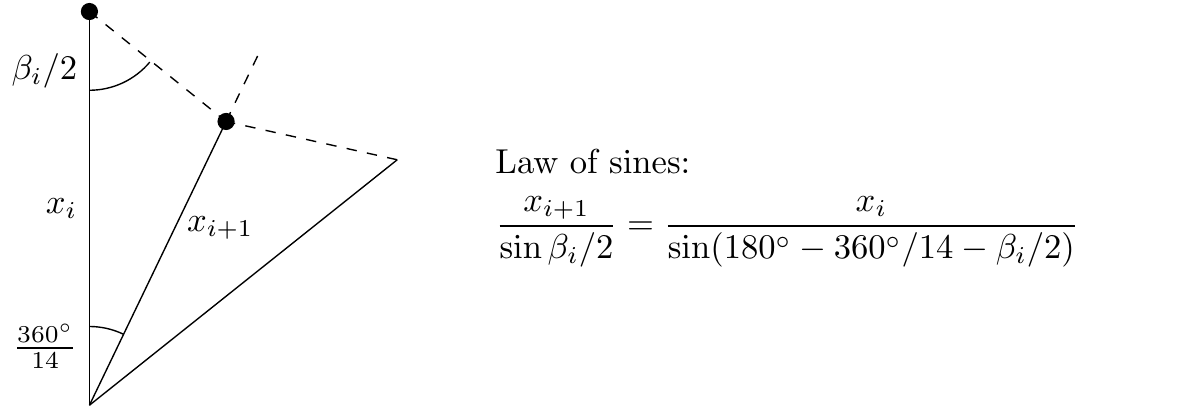}
	\caption{Finding the relation between $x_i$ and $x_{i+1}$.}
	\label{fig:lawofsines}
	\end{figure}
	
	By brute force calculation (for $\alpha_0=88/21$), we can verify the following:
	\begin{itemize*}
		\item $x_i<1$ for $1\leq i\leq 8$
		\item $x_9<0.7$
		\item $180\degrees<\alpha_9<190\degrees$
	\end{itemize*}
	We will now proceed to prove $x_i<1$ for $i>9$, using the fact that the values of $x_i$ are going through loops, at the end of which $x_i$ will have decreased. This can be formalized by the following claim:
	
	\textbf{Claim:} Let $N\geq 9$ such that $180\degrees<\alpha_N<190\degrees$.	Then the following is true for either $\ell=8$ or $\ell=9$:
	\begin{itemize*}
		\item $180\degrees<\alpha_{N+\ell}<190\degrees$
		\item $x_{N+j}\leq 1.3\cdot x_{N}$ for all $j=1,2,\dots,\ell-1$, and
		\item $x_{N+\ell}<0.96\cdot x_N$
	\end{itemize*}
	Note that the lemma follows from inductive application of this claim using $N=9$ as the base case since $1.3\cdot 0.7<1$ and since the sequence is \enquote{generally geometrically decreasing}. We do not need to prove $\alpha_{N+\ell}<190\degrees$ since this is always true for any $\alpha_i$ ($i\geq 1$), see lemma \ref{thm:anglecalc}.
	
	We will now prove the claim. Observe the following facts:
	\begin{itemize*}
		\item $180\degrees<\alpha_N<190\degrees$
		\item Therefore, $\alpha_{N+1}=\frac{2160\degrees}{7}-360\degrees+\alpha_N$, implying $120\degrees<\alpha_{N+1}<180\degrees$.
		\item Furthermore, the larger $\alpha_N$ is, the larger $\alpha_{N+1}$ will be.
		\item As long as $\alpha_{N+j}<180\degrees$, we can write
		\begin{align*}
			\alpha_{N+j+1}=\underbrace{\frac{2160\degrees}{7}-2\arccos(1/2-\cos(\alpha_{N+j}/2))}_{=:g(\alpha_{N+j})}
		\end{align*}
		We can therefore rewrite the sequence $\alpha_{N+1},\alpha_{N+2},\alpha_{N+3}\dots$ as
		\begin{align*}
			\alpha_{N+1},g(\alpha_{N+1}),g(g(\alpha_{N+1})),\dots
		\end{align*}
		This continues until, for some $\ell$, $\alpha_{N+\ell}=g^{(\ell-1)}(\alpha_{N+1})>180\degrees$. In other words, the behaviour of $\alpha_{N},\alpha_{N+1},\alpha_{N+2},\dots$ can be summarized as follows: It starts at $\alpha_N>180\degrees$, then $\alpha_{N+1}$ jumps below $180\degrees$ and the sequence $\alpha_{N+i}$ climbs back up until $\alpha_{N+\ell}>180\degrees$.
		\item We already established that the larger $\alpha_N$ is, the larger $\alpha_{N+1}$ will be. Also, $g(\alpha)$ is an increasing function for $120\degrees<\alpha<180\degrees$. Therefore, the larger $\alpha_N$ is, the larger each $\alpha_{N+j}$ will be for $j=1,\dots,\ell$.
	\end{itemize*}
	
	We observe the following two cases:
	
	\textbf{Case 1:} $183\degrees\leq\alpha_N<190\degrees$. Checking the extremal cases ($183\degrees$ and $190\degrees$), we can observe that either $\alpha_{N+8}>180\degrees$ or $\alpha_{N+9}>180\degrees$. We pick $\ell=8,9$ accordingly.
	
	\textbf{Case 2:} $180\degrees<\alpha_N<183\degrees$. Again checking the extremal cases ($183\degrees$ and $183\degrees$), we can observe that $\alpha_{N+9}$ will be the first angle above $180\degrees$. So we pick $\ell=9$.

	For both cases, note that $x_{N+j+1}=x_{N+j}f(\beta_{N+j})$. We can link the factor $f(\beta)$ to the angles $\alpha$ as follows
	\begin{itemize*}
		\item $\beta=2\cdot\arccos(1/2-\cos(\alpha/2))$ is a decreasing function of $\alpha$.
		\item $f(\beta)=\frac{\sin\beta/2}{\sin(1080\degrees/7-\beta/2)}$ is an increasing function of $\beta$, since $120\degrees<\beta<180\degrees$ as established in a previous lemma.
		\item Therefore, we can write $f(\beta(\alpha))=h(\alpha)$ and \textbf{as long as we underestimate $\alpha$, we overestimate $h(\alpha)$}.
	\end{itemize*}
	We write
	\begin{align*}
			x_{N+j}&=x_Nf(\beta_N)f(\beta_{N+1})f(\beta_{N+2})\cdots f(\beta_{N+j-1})\\
			&=x_Nh(\alpha_N)h(\alpha_{N+1})h(\alpha_{N+2})\cdots h(\alpha_{N+j-1})\\
			&=x_Nh(\alpha_N)h(\alpha_{N+1})h(g(\alpha_{N+1}))\cdots h(g^{(j-2)}(\alpha_{N+1}))
	\end{align*}
	Recall that $g$ is increasing and that $h$ is decreasing, so \textbf{as long as we underestimate $\alpha_N$ (and therefore also $\alpha_{N+1}$), we overestimate $x_{N+j}$}.
	
	We can return to the two cases:
		
	\textbf{Case 1 $183\degrees\leq\alpha_N<190\degrees$.} In this case we need $8$ or $9$ steps and $\alpha_N$ is at least $183\degrees$. By the above considerations, we can simply study the case $\alpha_N=183\degrees$. For larger starting values of $\alpha$, the values of $x$ can only be smaller. In that context, using the above equation:
		\begin{align*}
			x_{N+j}&\leq x_Nh(183\degrees)h(131\degrees)h(g(131\degrees))\cdots h(g^{(j-2)}(131\degrees))
		\end{align*}
	Now calculations yield that $h(183\degrees)h(131\degrees)h(g(131\degrees))\cdots h(g^{(j-2)}(131\degrees))$ will be less than $1.3$ for $j=1,2,\dots,8,9$ and less than $0.96$ for $j=8,9$. So $\ell=8$ or $\ell=9$ has the properties stated in the claim.
	
	\textbf{Case 2 $180\degrees<\alpha_N<183\degrees$.} In this case we need $9$ steps and $\alpha_N$ is at least $180\degrees$. By the above considerations, we can simply study the case $\alpha_N=180\degrees$ (as a limiting case, of course $\alpha=180\degrees$ never happens). For larger starting values of $\alpha$, the values of $x$ can only be smaller. In that context, using the above equation:
	\begin{align*}
			x_{N+j}&\leq x_Nh(180\degrees)h(128\degrees)h(g(128\degrees))\cdots h(g^{(j-2)}(128\degrees))
	\end{align*}
	Now calculations yield that $h(180\degrees)h(128\degrees)h(g(128\degrees))\cdots h(g^{(j-2)}(128\degrees))$ will be less than $1.3$ for $j=1,2,\dots,8,9$ less than $0.96$ for $j=9$. So $\ell=9$ has the properties stated in the claim.
	
	It is noteworthy that a split like the one at $183\degrees$ was necessary. In fact would we have a starting value of $\alpha_N=180\degrees$ but only $8$ steps, the factor would be greater than $1$. Hence the casework to show that this case doesn't happen.
	
	This finishes the proof of the claim and the lemma.
\end{proof}

This concludes the proof that the construction works, both in the symmetric case $\varphi=0$ and for small deviation $\varphi\in(-\epsilon,\epsilon)$.

\subsection{Analyzing the non-deviated construction}

Note that our goal was to construct a sequence of nets $G_i(0)$ such that

\begin{enumerate*}
	\item There are $14$ unbalanced vertices
	\item The number of balanced vertices goes to infinity as $i\longrightarrow\infty$.
	\item Some edges might intersect non-transversally (i.e. overlap).
\end{enumerate*}

The first observation is true as explained above: The only unbalanced vertices are the vertices on the outer ring as well as the seven vertices of the last layer that has been added.

Regarding the second observation, note that lemma \ref{thm:xlessthanone} demonstrates that the radius of the layers $V_i$ approaches zero as $i\longrightarrow\infty$. Therefore, increasing the number of layers increases the number of balanced vertices (otherwise, there would have to be a cyclical phenomenon in the construction, contradicting the fact that the radius goes to zero).

We now turn towards the third observation, which is what makes the introduction of deviation necessary. Note that the edges of the geodesic net can be categorized as follows:
\begin{itemize*}
	\item Whenever a new layer is added through the process of \emph{winging}, this adds 14 edges from $V_{i-1}$ to $V_i$ to the net. We call those the \emph{layer-connecting edges}. This includes the very first 14 edges to set up the outer circle and inner circle as seen in figure \ref{fig:circles}.
	\item Whenever we \emph{suspend} a vertex from a single hook, a single edge is being added. We call it the \emph{suspension edge}.
	\item Whenever we \emph{suspend} a vertex from two hooks, three edges are being added: two edges from two vertices of the \emph{outer circle} (the \emph{hooks}) to the Fermat point, as well as one edge from the Fermat point to the vertex that is being suspended -- we again call these \emph{suspension edges}.
\end{itemize*}

This now raises the question: Which edges can and will intersect non-transversally on $G_i(0)$ (i.e. if $\varphi=0$), either partially or in their entirety? We can observe (see also figure \ref{fig:star100}):
\begin{itemize*}
	\item Note that the Fermat point used in the process of suspension is the same for each layer in the symmetric case. Therefore, two of the three \emph{suspension edges} involving the same two hooks are the same from the hooks to the Fermat point, every single time these hooks are being used. The third suspension edge then always starts at the same Fermat point and continues radially to the vertex that is being suspended.
	\item Similarly, if we do a one-hook suspension, the suspension edge is always radial, so all the suspension edges going to the same hook have an overlay.
	\item Layer-connecting edges, on the other hand, can only intersect any other edge transversally: The Fermat points are outside the inner circle and therefore never meet the layer-connecting edges. Regarding the radial suspension edges, note that layer-connecting edges are never radial, so they are always transversal to radial edges. Besides that, note that layer-connecting edges always start and end on the boundary of a $260\degrees/14$ disk sector of the construction. So they could only be completely identical to another layer-connecting edge or intersect them transversally (if at all). For the sake of contradiction, assume that any layer-connecting edge would coincide with another layer-connecting edge. This would imply that the layer produced by them must be the same, including the incoming edges. Therefore, layers would repeat. This would contradict the phenomenon described in lemma \ref{thm:xlessthanone}, namely that the radius of the layers must converge to zero.
\end{itemize*}
As established previously, everything depends smoothly on the deviation $\varphi$. Therefore, any small deviation will maintain transversality where it already is given. Deviation will, however, have to make sure that we split up suspension edges.

\begin{figure*}
	\centering
	\includegraphics[width=0.9\textwidth]{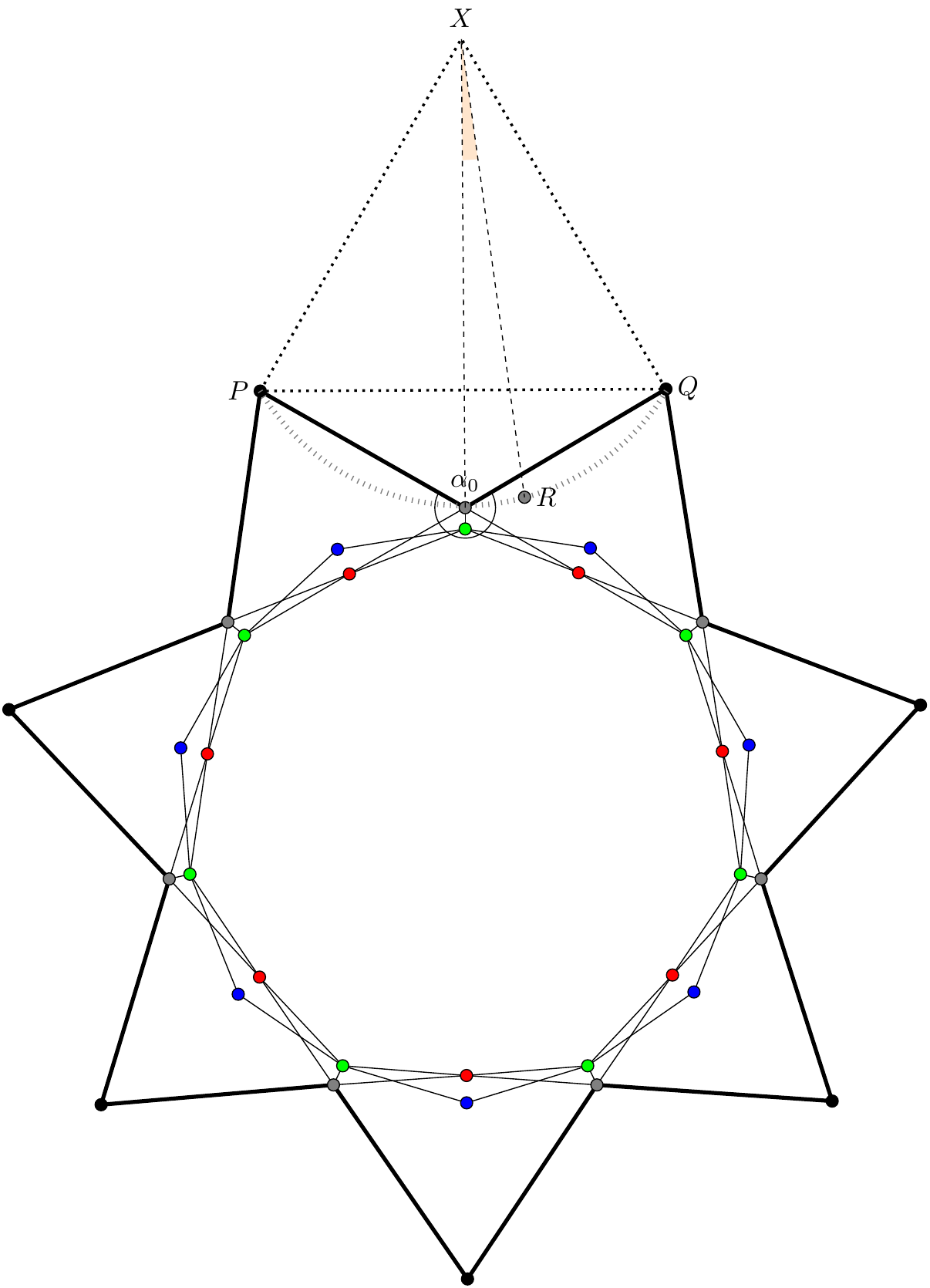}
	\caption{Note the dotted equilateral triangle $XPQ$ as well as the unique circular arc given by $P$, $R$ and $Q$. Denote by $\varphi$ the angle $OXR$ (where $O$ is the origin/center). The symmetric case is given by $\varphi=0$. \textbf{To introduce deviation}, we would slightly increase $\varphi$, and get a new position for $R$. This does not change the value of $\alpha_0$ (since we go along the circular arc). The same deviation is done at \emph{all} seven grey points. This means that the star remains rotationally symmetric under a rotation by $360\degrees/7$, but loses its symmetry under reflections.}
	\label{fig:stardeviate}
\end{figure*}

\subsection{Edges under deviation}

As just established above, we only need to be concerned with the edges produced using the method of \emph{suspension}. This section serves to support the following claim:

\textit{Whereas for $\varphi=0$, the geodesic nets $G_i(0)$ are highly symmetrical and many suspension edges overlap, for nonzero $\varphi\in(-\epsilon,\epsilon)$, all suspension edges of $G_i(\varphi)$ intersect transversally (if at all).}

\subsubsection{The sequence of suspension angles $\varphi_i$}

To demonstrate this, we will study the sequence of suspension angles, which is defined as follows, based on the two types of suspension that we employ:

\begin{mdframed}[style=greybox]
\begin{minipage}{0.7\textwidth}
\setlength{\parskip}{0.5em}
\begin{definition}[Suspension angle for one-hook suspensions, layers $V_i$ for $i$ odd]\label{def:varphiodd}
	Whenever we do a one-hook suspension for a layer $V_i$, we are connecting a vertex $v$ to the closest vertex on the outer circle.
	
	Denote the center of the outer circle by $O$ and the hook by $P$. Then we define the suspension angle $\varphi_i$ to be the angle $\angle OPv$. For clarification, consider the figure on the right giving a positive suspension angle.
	
	Note that $\varphi_i$ depends only on the deviation $\varphi$ (which is the only free parameter of our construction) and that $\varphi_i(0)=0$ for all layers.
	
	We can define this suspension angle for all odd layers, even though in some cases we don't need to suspend a vertex (if case A occurs).
\end{definition}
\end{minipage}\hfill
\begin{minipage}{0.25\textwidth}
	\raggedleft
	\includegraphics[width=\textwidth]{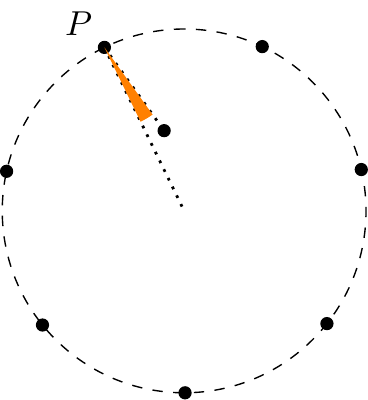}
\end{minipage}
\end{mdframed}

\vspace{3mm}
\begin{mdframed}[style=greybox]
\begin{minipage}{0.7\textwidth}
\setlength{\parskip}{0.5em}
\begin{definition}[Suspension angle for two-hook suspensions, layers $V_i$ for $i$ even]\label{def:varphieven}
	Whenever we do a two-hook suspension for a layer $V_i$, we are connecting a vertex $v$ to the closest two vertices on the outer circle $P$ and $Q$ through the Fermat point of $\Delta PvQ$, see the figure when defining two-hook suspension above).
	
	As before, we denote by $X$ the third vertex of the equilateral triangle $\Delta PQX$ used for the construction of the Fermat point. Let $O$ be the center of the outer circle. Then we define the suspension angle $\varphi_i$ to be the angle $\angle OXv$. For clarification, consider the figure on the right showing a positive suspension angle.
	
	Note that $\varphi_i$ depends only on the deviation $\varphi$ (which is the only free parameter of our construction) and that $\varphi_i(0)=0$ for all layers. Most importantly, for $i=0$ (the initial layer, aka the inner circle) the suspension angle is $\varphi_0(\varphi)=\varphi$.
	
	We can define this suspension angle for all even layers, even though in some cases we don't need to suspend a vertex.
\end{definition}
\end{minipage}\hfill
\begin{minipage}{0.25\textwidth}
	\raggedleft
	\includegraphics[width=\textwidth]{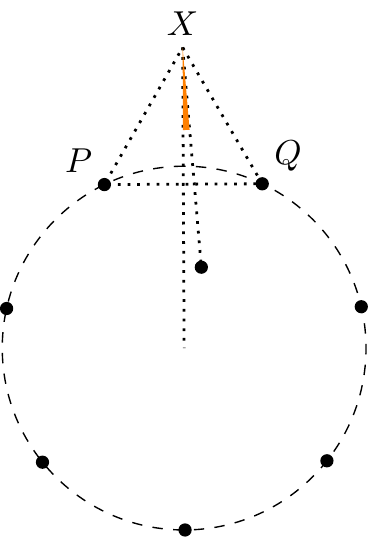}
\end{minipage}
\end{mdframed}

With this definition, we can now make the following observations:
\begin{fact}\label{fact:edgediff}
	Consider $\varphi\in(-\epsilon,\epsilon)$ a geodesic net $G_n(\varphi)$ as constructed above with layers $V_0,\dots,V_n$. Then
	\begin{itemize*}
		\item As established before, only suspension edges could overlap/intersect non-transversally.
		\item For all one-hook suspensions (odd layers), only the edges going from vertices $v_i$, $v_k$ of two different layers to the same hook $P$ could overlap. But as long as $\varphi_i\neq \varphi_k$, they will not do so (this is apparent from the figure in the definition above).
		\item For all two-hook suspensions (even layers), only the edges suspending vertices $v_i$, $v_k$ of two different layers from the same two hooks $P$ and $Q$ could overlap. But as long as $\varphi_i\neq \varphi_k$, they will not do so (see figure \ref{fig:fermatvarphi}).
	\end{itemize*}
\end{fact}

\begin{figure}[b]
  \centering
  \includegraphics[width=\textwidth]{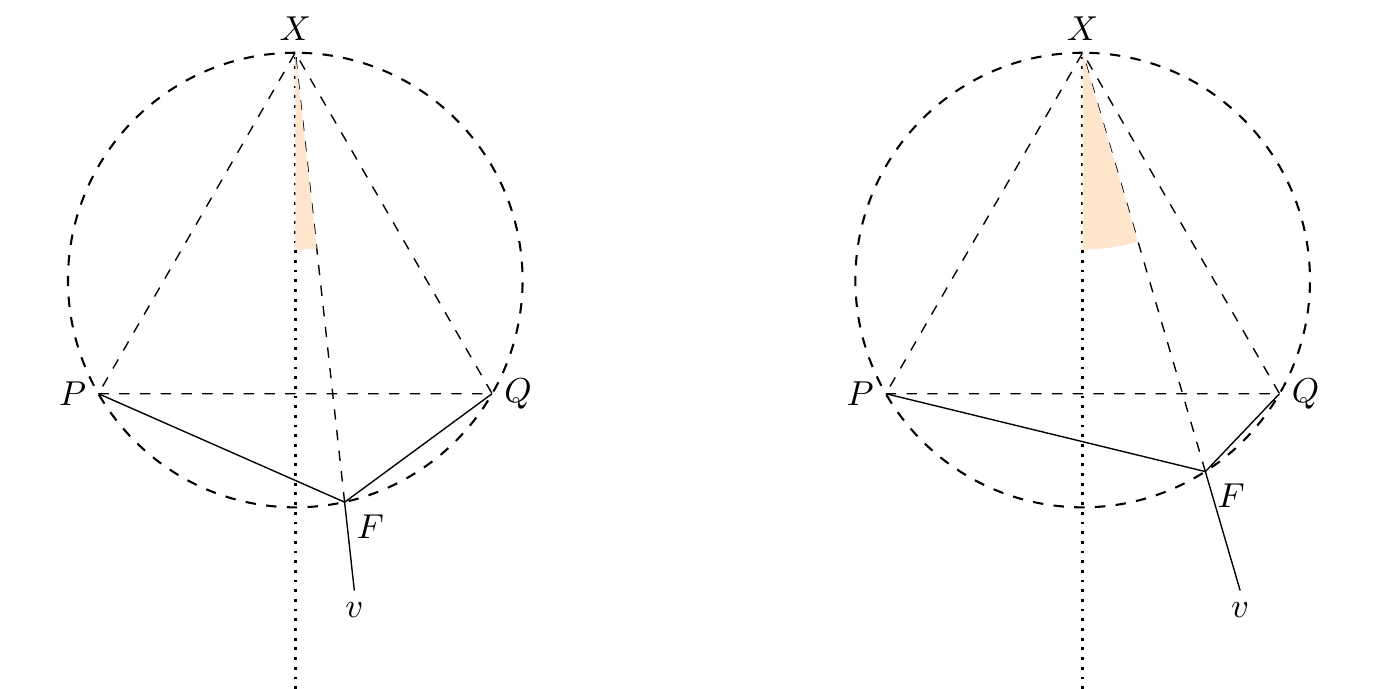}
  \caption{Construction of the Fermat point $F$ to suspend $v_i\in V_i$ from $P$ and $Q$. $\varphi_i$ is the angle between the axis of symmetry of $PQX$ and the segment $vX$. Observe: Whenever $\varphi_i$ (the marked suspension angle) is different, the segment $PF$ is at a different angle. The same is true for $QF$ and $vF$. So we only need to establish that $\varphi_i$ is different at every layer and this implies that none of the suspension edges overlap.}
  \label{fig:fermatvarphi}
\end{figure}

Note the following important observation:

\begin{fact}\label{fact:finitelayers}
	$G_n(\varphi)$ consists of finitely many layers, therefore $\varphi_0,\varphi_1,\dots,\varphi_n$ is a \emph{finite} sequence.	
\end{fact}

Based on definitions \ref{def:varphiodd}, \ref{def:varphieven} and fact \ref{fact:edgediff}, we arrive at the following lemma:

\begin{lemma}\label{def:varphimustbedifferent}
	If for any fixed $\varphi\in(-\epsilon,\epsilon)$, the sequence $\varphi_i(\varphi)$ never repeats itself, all edges of the resulting geodesic net $G_n(\varphi)$ intersect transversally (if at all). In other words, there are no edges with weight other than one.
\end{lemma}
Note that, in fact it would be enough if the $\varphi_i$ are different for the same parity (since even and odd layers never have suspension edges in common).

Based on symmetry (see also figure \ref{fig:star100}), we can observe

\begin{fact}
	$\varphi_i(0)=0$ for all $i$.
\end{fact}

Also, since we established smooth dependence of the construction on the deviation $\varphi$ before and since this is in fact the only free parameter, we can consider the derivative of $\varphi_i(\varphi)$. We make the following conjecture:

\begin{conjecture}\label{conj:varphiconjecture}
	$\varphi_i'(0)$ is a sequence that never repeats itself.
\end{conjecture}

Keeping in mind that $\varphi_i$ is a finite sequence, the previous fact and conjecture (i.e. same value at $0$, but different derivatives) would then imply the following:

\begin{conjecture}
	For small nonzero $\varphi\in(-\epsilon,\epsilon)$, the sequence $\varphi_i(\varphi)$ never repeats itself. This implies that $G_n(\varphi)$ is a geodesic net for which all edges have weight one.
\end{conjecture}

So this $G_n(\varphi)$ would in fact fulfill all required conditions:

\begin{enumerate*}
	\item There are $14$ unbalanced vertices.
	\item For $n$ large enough, we can achieve an arbitrarily large number of balanced vertices.
	\item All edges have weight one.
\end{enumerate*}

\subsection{Studying the sequence $\varphi_i'(0)$}

So we are left with conjecture \ref{conj:varphiconjecture}. For the remainder, we will consider statements we can make about the sequence $\varphi_i'(0)$.

We will do the following:
\begin{itemize*}
	\item We will provide an explicit recursive formula for $\varphi_i'(0)$.
	\item We will present numerical results that strongly suggest that conjecture \ref{conj:varphiconjecture} is true.
\end{itemize*}

\subsection{An explicit formula for $\varphi_i'(0)$}

In the following, all derivatives will be with respect to $\varphi$. First note, that $\varphi_0=\varphi$ and therefore obviously $\varphi_0'(0)=1$. We will now find a recursive formula for $\varphi_i'(0)$. The crucial question is how $\varphi_{i+1}'(0)$ depends on $\varphi_i'(0)$.

\begin{figure*}[b]
	\centering
	\includegraphics[width=0.45\textwidth]{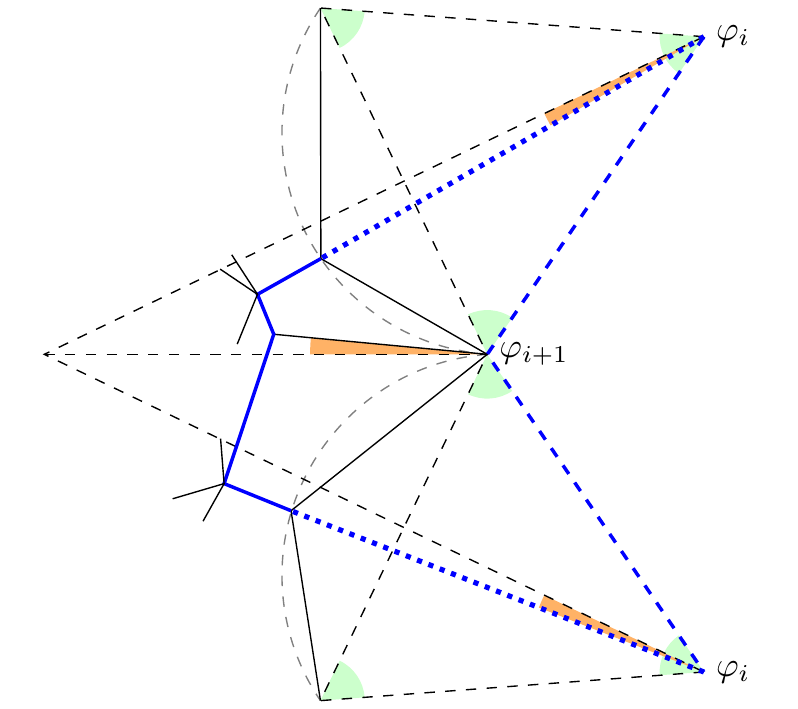}
	\includegraphics[width=0.45\textwidth]{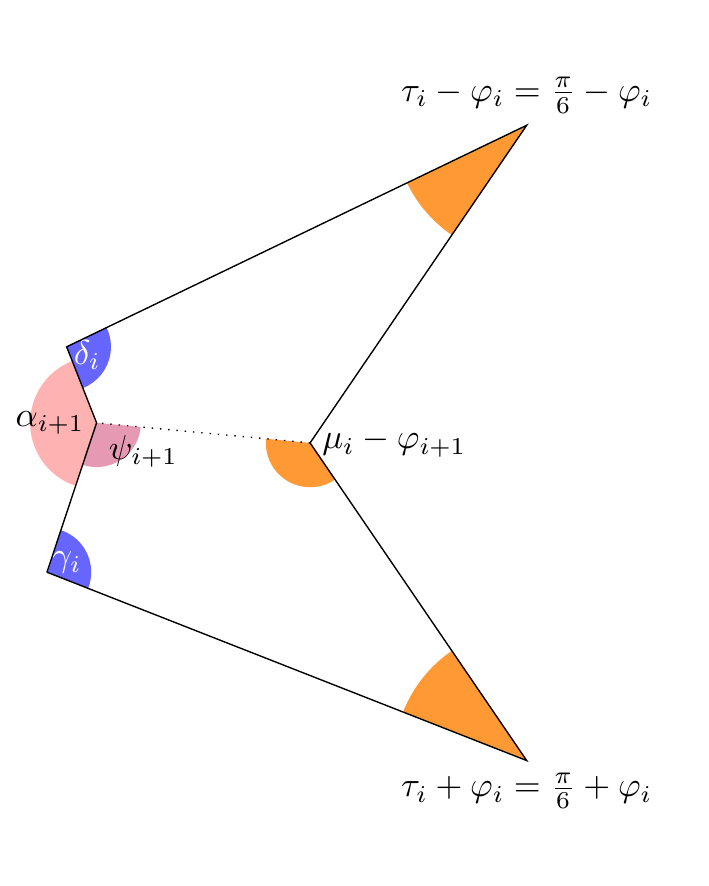}
	
	\includegraphics[width=0.45\textwidth]{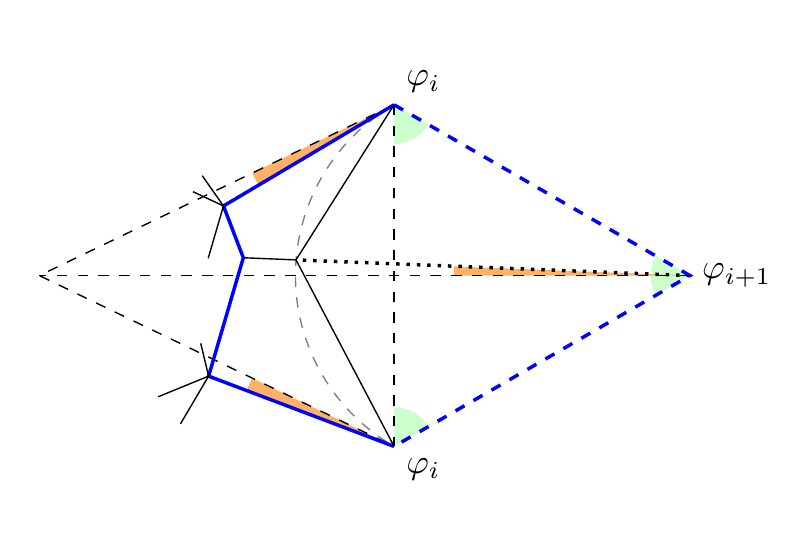}
	\includegraphics[width=0.45\textwidth]{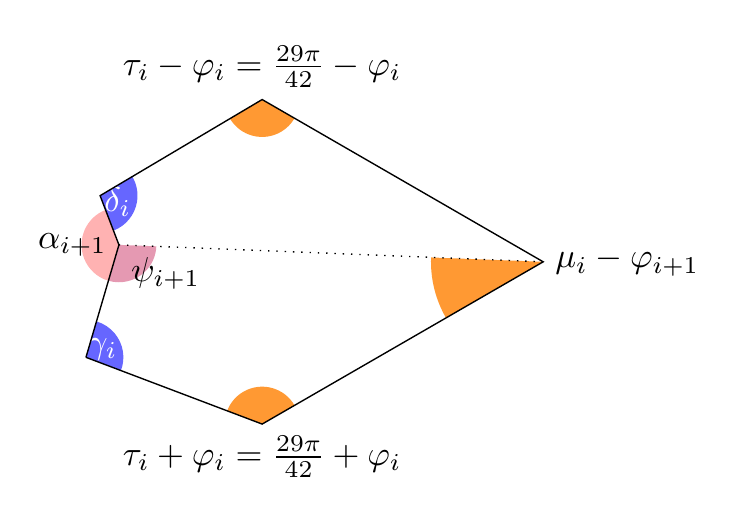}	
	
	\caption{We want to relate $\varphi_{i+1}'$ to $\varphi_i'$. The upper pictures are considering the case $i$ even, the lower pictures are considering the case $i$ odd. All green angles are $60\degrees$. $\mu$ is a constant angle. All dashed lines are stationary, all dotted and solid lines vary over a change of the deviation $\varphi$. The picture on the right extracts the blue hexagon out of the left picture. Note that in this example, $\varphi_i$ is positive whereas $\varphi_{i+1}$ is negative}
	\label{fig:varphistep}	
\end{figure*}

Figure \ref{fig:varphistep} shows the two cases for $i$ even and $i$ odd. Note that if $\varphi=0$, then the picture is symmetric along a horizontal reflection. In both cases, we get a hexagon made of two quadrilaterals.

In either case, by the angle sum in the lower quadrilateral:
\begin{align*}
	\tau_i+\varphi_i+\gamma_i+\psi_{i+1}+\left(\mu_{i}-\varphi_{i+1}\right)=360\degrees
\end{align*}
Note that $\tau_i$ and $\mu_i$ are constants that don't change under $\varphi$. Therefore, differentiation by $\varphi$ leads to:
\begin{align*}
	\varphi_i'+\gamma_i'+\psi_{i+1}'-\varphi_{i+1}'=0\Rightarrow \varphi_{i+1}'=\varphi_i'+\gamma_i'+\psi_{i+1}'
\end{align*}

In the appendix, we will establish the following relationships at $\varphi=0$.
\begin{align*}
	\varphi_0'&=1
	\qquad\qquad 
	\psi_0'=\frac12-\sin(\pi/6-\alpha_0)\\[2mm]
	\varphi_{i+1}'&=\varphi_i'+\gamma_i'+\psi_{i+1}'\qquad\qquad
	\psi_{i+1}'=b_i\gamma_i'+a_i\varphi_i'\qquad\qquad
	\gamma_i'=c_i\psi_i'
\end{align*}
The formulas for the coefficients make use of $x_i$ and $\alpha_i$, which were defined previously:
\begin{align*}
	a_i&=-\tan\frac{\alpha_{i+1}}{2}\left(\frac{\sin(\pi/7+\sigma_{i+1})\sin(\pi/7+\sigma_{i+1}+\tau_i)}{\sin\tau_i}+\frac{\frac12
	-\frac{\sin(\tau_i+2\pi/7+2\sigma_{i+1})}{2\sin\tau_i}}{\tan(\frac{\alpha_{i+1}}{2}-\sigma_{i+1})}\right)\\[2mm]
	b_i&=-\frac{\tan\alpha_{i+1}/2}{\tan(\alpha_{i+1}/2-\sigma_{i+1})}
	\qquad\qquad
	c_i=\begin{cases}
	-1&\alpha_i>180\degrees\\
	\displaystyle\frac{2\cos\alpha_i/2}{1-2\cos\alpha_i/2}&\alpha_i<180\degrees
\end{cases}\\[2mm]
	\tau_i&=\begin{cases}
		\pi/6&i\text{ even}\\
		29\pi/42&i\text{ odd}\\
	\end{cases}\qquad\qquad
	\sigma_{i+1}=\begin{cases}\displaystyle
		\arctan\frac{\sin\pi/7}{\frac{\sin\alpha_0/2}{x_i\sin(\pi/7+\alpha_0/2)}-\cos\pi/7}&i\text{ even}\\
		\displaystyle
		\arctan\frac{\sin\pi/7}{\frac{\sin\alpha_0/2}{x_i\sin(\pi/7+\alpha_0/2)}\frac{\sin(\pi/7+\pi/6)}{\sin\pi/6}-\cos\pi/7}&i\text{ odd}
	\end{cases}
\end{align*}

\subsection{Numerical consideration of the sequence $\varphi_i'(0)$}

While the formulas above are all explicit, they are arguably not very \enquote{handy} which makes understanding their behaviour a challenging task. Recall that all we need is that $\varphi_i'(0)$ never repeats. This would then imply that a small deviation from $G_n(0)$ to $G_n(\varphi)$ would in fact split up all edges as required.

For a better understanding, we used \textsc{MATLAB} to compute the first items of the sequence. Figure \ref{fig:phiplot} shows the first 100 elements of the sequence $\varphi_i'(0)$, on a logarithmic scale. These calculations lead to the following observations, which in turn support conjecture \ref{conj:varphiconjecture}, saying that $\varphi_i'(0)$ doesn't repeat:
\begin{itemize*}
	\item The magnitude of the sequence grows exponentially.
	\item The sequence seems to be generally increasing (i.e. increasing with a small number of exceptions)
	\item Since it is enough if the sequence differs for all even $i$ and for all odd $i$, we get additional \enquote{leeway}.
\end{itemize*}

In fact, numerical evidence suggests that the first 100 elements of the sequence do not repeat. We computed the first 100 elements of $\varphi'_i(0)$ with MATLAB using variable precision arithmetic, using between 10 and 100 significant digits. The following result remained stable under variable precision:
\begin{align*}
	\min_{i\neq j}|\varphi'_i(0)-\varphi'_j(0)|\approx 3.743673268
\end{align*}
This minimum is realized by $\varphi'_0(0)$ and $\varphi'_2(0)$.

\begin{figure}
  \centering
  \includegraphics{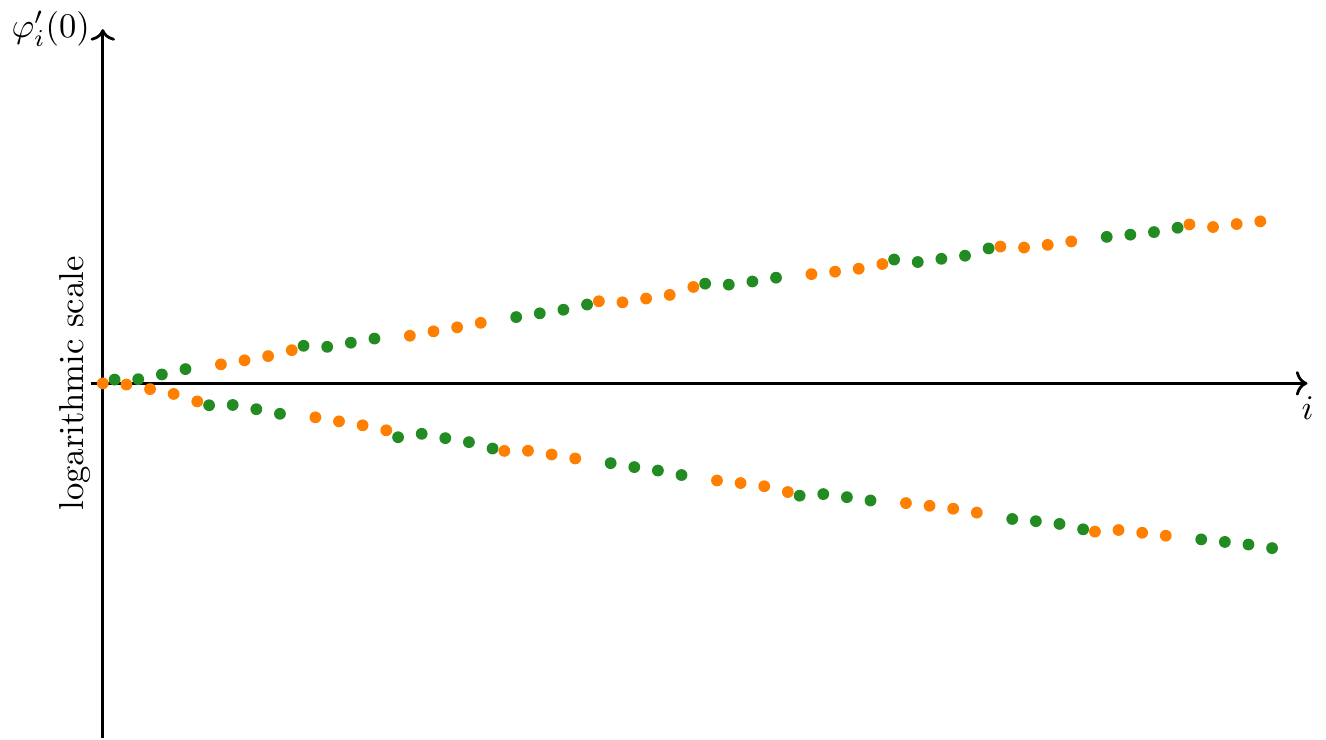}
  \caption{The first 100 values of $\varphi_i'(0)$ on a \emph{logarithmic} scale, calculated with $\alpha_0=88/21$. All points for $i$ even are marked in orange. All points for $i$ odd are marked in green.}
  \label{fig:phiplot}
\end{figure}

\section{Appendix: Finding the formulas for the sequence $\varphi_i'(0)$}

Consider figure \ref{fig:varphistep} and the following formula we derived previously:
\begin{align*}
	\varphi_{i+1}'=\varphi_i'+\gamma_i'+\psi_{i+1}'
\end{align*}
In this appendix, we intend to do the following:
\begin{itemize*}
	\item Find the starting values of $\varphi_0'(0)$ and $\psi_0'(0)$.
	\item Establish that $\alpha_i'(0)=0$ for all $i$.
	\item Derive a formula for $\gamma_i'(0)$ in terms of $\psi_i'(0)$.
	\item Derive a formula for $\psi_{i+1}'(0)$ in terms of $\psi_i'(0)$ and $\varphi_i'(0)$.
\end{itemize*}

\subsection{Starting values}

Since $\varphi$ is defined to be the \emph{suspension angle} for the inner circle, which is the zeroth level (compare definition \ref{def:varphieven} and compare with the initial definition of the inner circle), $\varphi_0=\varphi$ and therefore
\begin{align*}
	\varphi_0'(0)=1
\end{align*}

\begin{figure}
  \centering
  \includegraphics[width=0.6\textwidth]{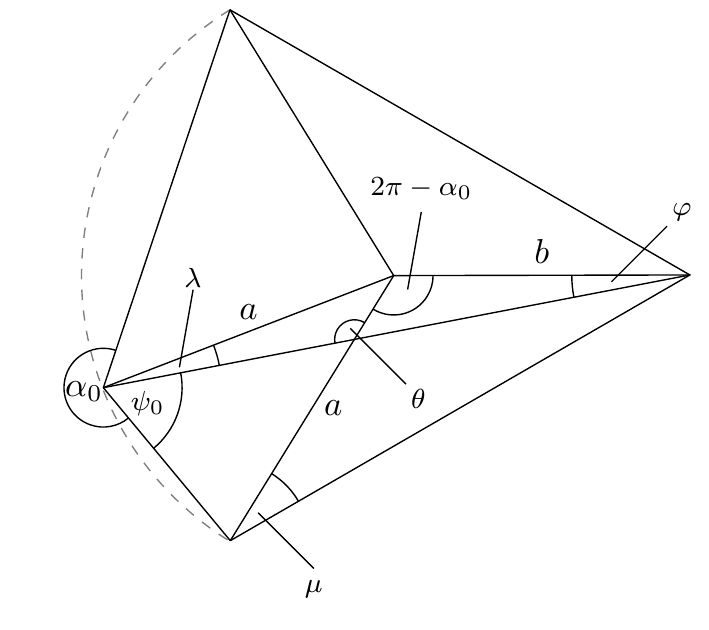}
  \caption{Dependence of $\psi_0$ on $\varphi$}
  \label{fig:psiprimeinitial}
\end{figure}

To find $\psi_0'(0)$, consider figure \ref{fig:psiprimeinitial}, more specifically the isosceles triangle with two sides of length $a$. Using angle sums in triangles:
\begin{align*}
	\psi_0=(\pi-(\pi-\lambda-\theta))/2-\lambda=\theta/2-\lambda/2\ \Rightarrow\ \psi_0'=\theta'/2-\lambda'/2
\end{align*}

\textbf{Finding $\lambda'$.} The law of sines, applied to the two triangles that share the common side of length $b$ but have two different sides of length $a$ gives us
\begin{align*}
	\frac{\sin\lambda}{\sin\varphi}=\frac{b}{a}=\frac{\sin\mu}{\sin\pi/6}=2\sin\mu\ \Rightarrow\ \sin\lambda=2\sin\mu\sin\varphi\ \Rightarrow\ \cos\lambda\ \lambda'=2\sin\mu\cos\varphi
\end{align*}
At $\varphi=0$ we also have $\lambda=0$ and therefore $\lambda'(0)=2\sin\mu=2\sin(\pi-\pi/6-(2\pi-\alpha_0))=2\sin(\pi/6-\alpha_0)$.

\textbf{Finding $\theta'$.} Using the angle sum in triangles, we get $\theta+\mu+\pi/6-\varphi=\pi$ and therefore $\theta'=1$.

\textbf{Combining these} we arrive at $\psi_0'(0)=\frac12-\sin(\pi/6-\alpha_0)$.

\subsection{$\alpha_i'(0)$ and $\gamma_i'(0)$}

Recall that $\alpha_i$ is the interior angle at the vertices $V_i$ of the 14-gon formed by $V_{i-1}$ and $V_i$. It is one of the two angles between the incoming edges at the vertex of a layer, before winging (see also methods \ref{meth:wing2} and \ref{meth:wing3} as well as the formulas in the proof of \ref{thm:anglecalc}). Like all other angles, each $\alpha_i=\alpha_i(\varphi)$ is a smooth function of the deviation angle $\varphi$. Figures \ref{fig:varphistep} and \ref{fig:alphapsigamma} show how $\gamma$ is one of the angles between an outgoing edge and the suspension edge, whereas $\psi$ is one of the angles between an incoming edge and the suspension edge (even if we do not suspend, as is the case for a degree $4$ vertex where $\alpha>180\degrees$, we can still consider the angle compared to a hypothetical suspension edge.

\begin{lemma}\label{lem:alphagammaprops}
	For all $i=0,1,2,\dots$, we have the following relationships at $\varphi=0$
	\begin{enumerate*}
		\item $\alpha_i'=0$
		\item $\gamma_i'=c_i\psi_i'$ where $c_i=\begin{cases}
	-1&\alpha_i>180\degrees\\
	\displaystyle\frac{2\cos\alpha_i/2}{1-2\cos\alpha_i/2}&\alpha_i<180\degrees
\end{cases}$
	\end{enumerate*} 
\end{lemma}
	
It is important to emphasize that these relationships between derivatives only hold at $\varphi=0$, which is, however, enough for us.

	\begin{figure}[b]
  \centering
  \includegraphics[width=0.6\textwidth]{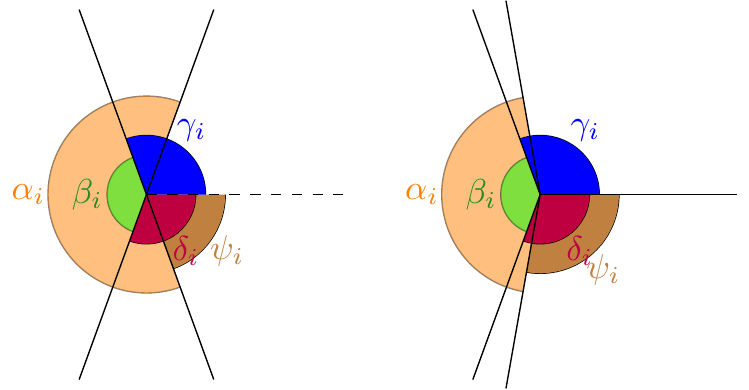}
  \caption{The angle relationships at a vertex of $V_i$ in the case of $\alpha_i>180\degrees$ (left, winging a degree 2 vertex) and in the case of $\alpha_i<180\degrees$ (right, winging a degree 3 vertex)}
  \label{fig:alphapsigamma}
\end{figure}

\begin{proof}
	\begin{enumerate*}
	\item For the base case, note that the inner circle is chosen precisely to ensure that $\alpha_0=88/21$ for any choice of $\varphi$. So $\alpha_0(\varphi)$ is constant and the base case follows.
	
	For the induction step, recall the following formulas (see proof of \ref{thm:anglecalc}) for the symmetric case $\varphi=0$:
	\begin{align*}
		\alpha_{i+1}&=\frac{12\cdot180\degrees}{7}-\beta_i\\
		\beta_{i}&=\begin{cases}
			360\degrees-\alpha_i&\alpha_i>180\degrees\quad\text{(winging of degree 2 vertex)}\\
			2\cdot\arccos(1/2-\cos(\alpha_i/2))&\alpha_i<180\degrees\quad\text{(winging of degree 3 vertex)}\\
		\end{cases}
	\end{align*}
	It is worth checking which of these formulas still apply in the deviated case $\varphi\neq 0$.
	\begin{itemize*}
		\item The relation $\alpha_{i+1}=\frac{12\cdot180\degrees}{7}-\beta_i$ remains unchanged, since the formula is based on the angle sum in the 14-gons during construction. Therefore $\alpha_{i+1}'(0)=-\beta_i'(0)$.
		\item The formula for $\beta_i$ in the case of $\alpha_i>180\degrees$ also applies to the non-symmetric case (see method $\ref{meth:wing2}$). It follows that $\beta_i'(0)=-\alpha_i'(0)=0$ by induction hypothesis and we are done.
		\item The formula for $\beta_i$ in the case of $\alpha_i<180\degrees$, however, cannot be used for the asymmetric situation (as explained in method \ref{meth:wing3}, it only works for the symmetric case).
	\end{itemize*}
	So we are left to show that $\beta_i'(0)=0$ assuming that $\alpha_i'(0)=0$ and $\alpha_i<180\degrees$, but we can't use the given formula.
	
	Instead, lets consider the right of figure \ref{fig:alphapsigamma}, depicting this case. We ran rotate the vertex as depicted so that one edge is pointing to the right. Note that in the symmetric case, i.e. when $\varphi=0$, the picture is symmetric under reflection along the horizontal axis. Generally, though, this is not the case.
	
	The vertex is balanced, so we have
	\begin{align*}
		&[1,0]+[\cos\gamma_i,\sin\gamma_i]+[\cos\delta_i,-\sin\delta_i]+[\cos\psi_i,-\sin\psi_i]+[\cos(\psi_i+\alpha_i),-\sin(\psi_i+\alpha_i)]=[0,0]\\
		\Leftrightarrow&\begin{cases}
	1+\cos\gamma_i+\cos\delta_i+\cos\psi_i+\cos(\psi_i+\alpha_i)=0\\
	\sin\gamma_i-\sin\delta_i-\sin\psi_i-\sin(\psi_i+\alpha_i)=0
	\end{cases}
	\end{align*}
	We can derive everything with respect to $\varphi$ and arrive at
	\begin{align*}
		-\gamma_i'\sin\gamma_i-\delta_i'\sin\delta_i-\psi_i'\sin\psi_i-(\psi_i'+\alpha_i')\sin(\psi_i+\alpha_i)=0\\
		\gamma_i'\cos\gamma_i-\delta_i'\cos\delta_i-\psi_i'\cos\psi_i-(\psi_i'+\alpha_i')\cos(\psi_i+\alpha_i)=0
	\end{align*}
	We are concerned with the derivatives at $\varphi=0$. As pointed out, the picture is symmetric in that case and we get $\gamma_i=\delta_i$ as well as $\psi_i+\alpha_i=2\pi-\psi_i$. By induction hypothesis, we also have $\alpha_i'=0$. So we can simplify to
	\begin{align*}
		-\gamma_i'\sin\gamma_i-\delta_i'\sin\gamma_i-\psi_i'\sin\psi_i-\psi_i'\sin(2\pi-\psi_i)=0\\
		\gamma_i'\cos\gamma_i-\delta_i'\cos\gamma_i-\psi_i'\cos\psi_i-\psi_i'\cos(2\pi-\psi_i)=0
	\end{align*}
	We can simplify further to
	\begin{align*}
		-\gamma_i'\sin\gamma_i-\delta_i'\sin\gamma_i-\psi_i'\sin\psi_i+\psi_i'\sin\psi_i=0\\
		\gamma_i'\cos\gamma_i-\delta_i'\cos\gamma_i-\psi_i'\cos\psi_i-\psi_i'\cos\psi_i=0
	\end{align*}
	Note that $\gamma_i$ can't be a multiple of $\pi$ at $\varphi=0$ since that would imply $\beta_i=0$ or $\beta_i=2\pi$ which is never the case as established previously. So $\sin\gamma_i\neq0$ and the two equations do in fact simplify to
	\begin{align*}
		\gamma_i'+\delta_i'&=0\\
		(\gamma_i'-\delta_i')\cos\gamma_i&=2\psi_i'\cos\psi_i
	\end{align*}
	As is clear from figure \ref{fig:alphapsigamma}, $\beta_i=2\pi-(\gamma_i+\delta_i)$. Therefore $\beta_i'=-(\gamma_i'+\delta_i')=0$ as required.
	\item If $\alpha_i>180\degrees$ consider the left of figure \ref{fig:alphapsigamma} from which it is clear that $\gamma_i=180\degrees-\psi_i$. It follows that $\gamma_i'=-\psi_i'$ as stated (this relationship, in fact, would be true for any $\varphi\in(-\epsilon,\epsilon)$, not just for $\varphi=0$).
	
		If $\alpha_i<180\degrees$, we have just established the following at $\varphi=0$
		\begin{align*}
			\gamma_i'+\delta_i'=0\text{ and }
			(\gamma_i'-\delta_i')\cos\gamma_i=2\psi_i'\cos\psi_i
			\quad\Rightarrow\quad
			2\gamma_i'\cos\gamma_i=2\psi_i'\cos\psi_i
			\quad\Rightarrow\quad
			\gamma_i'=\frac{\cos\psi_i}{\cos\gamma_i}\psi_i'
		\end{align*}
		Furthermore, again due to symmetries at $\varphi=0$
		\begin{align*}
			\cos\gamma_i&=\cos(\pi-\beta_i/2)=-\cos(\beta_i/2)=-\cos\left(\frac{2\arccos(1/2-\cos(\alpha_i/2))}{2}\right)=-(1/2-\cos\alpha_i/2)\\
			\cos\psi_i&=\cos(\pi-\alpha_i/2)=-\cos\alpha_i/2
		\end{align*}
		Combining these equations, we arrive at $\gamma_i'=c_i\psi_i'$ with $c_i=\displaystyle\frac{2\cos\alpha_i/2}{1-2\cos\alpha_i/2}$.
	\end{enumerate*}
\end{proof}

\subsection{$\tau_i$ and $\sigma_i$}

$\tau_i$ is the \enquote{reference angle} from which the suspension angle $\varphi_i$ is measured. Consider the $14$-gon formed by the two outer sides of each of the seven equilateral triangles built on the outer circle. The interior angle at seven corners is $\pi/3$, the interior angle at the other seven corners is $29\pi/21$. These are the values for $2\cdot\tau_i$ even and odd respectively.

The angle $\sigma_{i+1}$ is depicted in figure \ref{fig:phi12details}. It changes under deviation, but at $\varphi=0$ it can be directly calculated from the sequence of $x_i$. If $i$ is even, consider figure \ref{fig:sigmaatzero}. Note that $\sigma_{i+1}$ is one of the angles in a triangle with side $x_i$ and angle $\pi/7$. Since $\sigma_{i+1}$ is the angle at a vertex of the \emph{outer circle}, we know one more side of the triangle. Recall that the inner circle is fixed at radius $1$ and consider figure \ref{fig:circles}, giving the other side as $\frac{\sin\alpha_0/2}{\sin(\pi/7+\alpha_0/2)}$. Using the law of sines, we get:
\begin{align*}
	\sin\sigma_{i+1}=x_i\sin(\sigma_{i+1}+\pi/7)\frac{\sin(\pi/7+\alpha_0/2)}{\sin\alpha_0/2}
\end{align*}
If $i$ is odd, a similar argument yields
\begin{align*}
	\sin\sigma_{i+1}=x_i\sin(\sigma_{i+1}+\pi/7)\frac{\sin(\pi/7+\alpha_0/2)}{\sin\alpha_0/2}\frac{\sin\pi/6}{\sin(\pi/7+\pi/6)}
\end{align*}
which is based on the fact that the additional vertex of $\sigma_i$ is at radius $\frac{\sin\alpha_0/2}{\sin(\pi/7+\alpha_0/2)}\frac{\sin(\pi/7+\pi/6)}{\sin\pi/6}$.

Since $\sigma_{i+1}<90\degrees$, we can solve both equations for $\sigma_{i+1}$ and get
\begin{align*}
	\sigma_{i+1}=\begin{cases}\displaystyle
		\arctan\frac{\sin\pi/7}{\frac{\sin\alpha_0/2}{x_i\sin(\pi/7+\alpha_0/2)}-\cos\pi/7}&i\text{ even}\\
		\displaystyle
		\arctan\frac{\sin\pi/7}{\frac{\sin\alpha_0/2}{x_i\sin(\pi/7+\alpha_0/2)}\frac{\sin(\pi/7+\pi/6)}{\sin\pi/6}-\cos\pi/7}&i\text{ odd}
	\end{cases}
\end{align*}

\subsection{The formula for $\psi_i'(0)$}

We are left to establish the following

\begin{lemma}
	For all $i=0,1,2,\dots$ at $\varphi=0$, we have $\psi_{i+1}'=b_i\gamma_i'+a_i\varphi_i'$ where
	\begin{align*}
		b_i&=-\frac{\tan\frac{\alpha_{i+1}}{2}}{\tan(\frac{\alpha_{i+1}}{2}-\sigma_{i+1})}\qquad\qquad\\
		a_i&=-\tan\frac{\alpha_{i+1}}{2}\left(\frac{\sin(\pi/7+\sigma_{i+1})\sin(\pi/7+\sigma_{i+1}+\tau_i)}{\sin\tau_i}+\frac{\frac12
	-\frac{\sin(\tau_i+2\pi/7+2\sigma_{i+1})}{2\sin\tau_i}}{\tan(\frac{\alpha_{i+1}}{2}-\sigma_{i+1})}\right)
	\end{align*}
\end{lemma}

\begin{figure}
  \centering
	\includegraphics{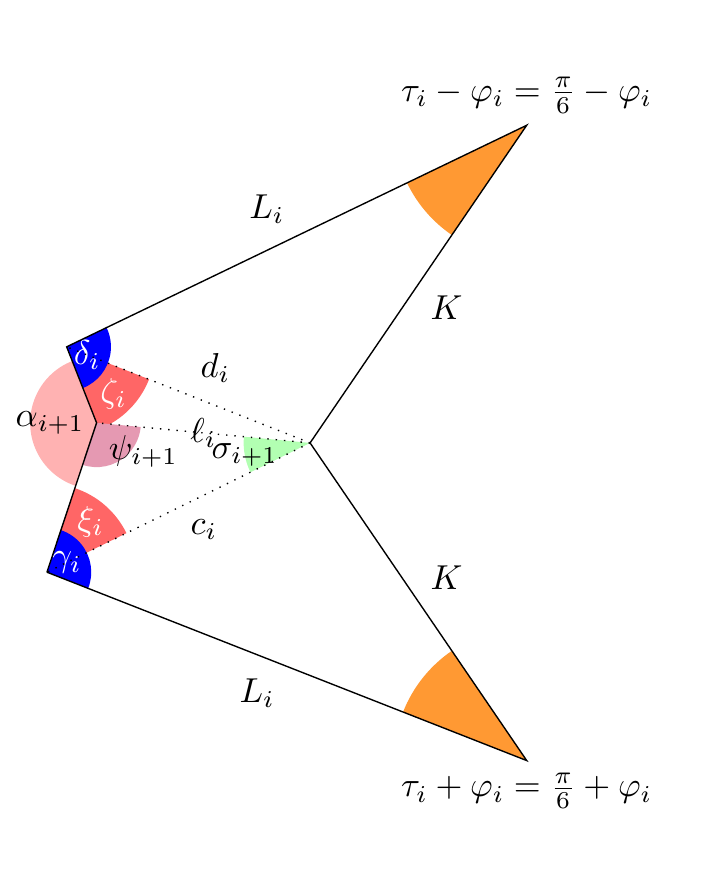}
  \caption{A detailed overview of a part of figure \ref{fig:varphistep}}
  \label{fig:phi12details}
\end{figure}

\begin{proof}
Consider the two cases in figure \ref{fig:varphistep}. We will cover the first case in detail ($i$ even). The other case, as given the lower of the two figures, can be deduced in the same way (the pictures only differ in the size of the angles, the underlying setup of polygons is the same). For this first case, figure \ref{fig:phi12details} gives a more detailed overview.

$\psi_i$ is a particular angle in the given hexagon. The angles $\kappa_i$ and $\tau_i$ don't change under $\varphi$ and neither does $K$ (it is the length of the sides of the equilateral triangles over the outer circle). Note the following: If we consider these angles and lengths
\begin{itemize*}
	\item The lengths $L_i$
	\item The angles $\varphi_i$, $\gamma_i$, $\delta_i$
\end{itemize*}
then there is a one-to-one correspondence (within an open set) between tuples $(L_i,\varphi_i,\gamma_i,\delta_i)$ and hexagons.

Also note that the pictures is symmetric at $\varphi=0$, implying $\gamma_i(0)=\delta_i(0)$ and $\varphi_i(0)=0$. Since these quantities uniquely define the hexagon and since $\psi_{i+1}$ is an angle defined by the hexagon, we arrive at
\begin{align*}
	\psi_{i+1}'=\frac{\partial \psi_{i+1}}{\partial L_i}L_i'+\frac{\partial \psi_{i+1}}{\partial \gamma_i}\gamma_i'+\frac{\partial \psi_{i+1}}{\partial \delta_i}\delta_i'+\frac{\partial \psi_{i+1}}{\partial \varphi_i}\varphi_i'
\end{align*}
As usual, we consider these derivatives at $\varphi=0$. The proof of the previous lemma established that in that case $\gamma_i'=-\delta_i'$. Also note that in the symmetric picture, a change in $L_i$ will not affect $\psi_i$. Therefore $\frac{\partial \psi_{i+1}}{\partial L_i}$ at $\varphi=0$. Combining all these, we arrive at
\begin{align*}
	\psi_{i+1}'=\left(\frac{\partial \psi_{i+1}}{\partial \gamma_i}-\frac{\partial \psi_{i+1}}{\partial \delta_i}\right)\gamma_i'+\frac{\partial \psi_{i+1}}{\partial \varphi_i}\varphi_i'
\end{align*}
So we are left to show that
\begin{enumerate*}
	\item $b_i=\left(\frac{\partial \psi_{i+1}}{\partial \gamma_i}-\frac{\partial \psi_{i+1}}{\partial \delta_i}\right)=-\frac{\tan\frac{\alpha_{i+1}}{2}}{\tan(\frac{\alpha_{i+1}}{2}-\sigma_{i+1})}$
	\item $a_i=\frac{\partial \psi_{i+1}}{\partial \varphi_i}=-\frac{\tan\frac{\alpha_{i+1}}{2}}{\sin\tau_i}\left(\sin(\pi/7+\sigma_{i+1})\sin(\pi/7+\sigma_{i+1}+\tau_i)+\frac{-\sin\tau_i
	+\sin(\tau_i+2\pi/7+2\sigma_{i+1})}{2\tan(\frac{\alpha_{i+1}}{2}-\sigma_{i+1})}\right)$
\end{enumerate*}
To do so, we will make extensive use of trigonometric identities and laws.

The law of sines provides
\begin{align*}
	&c_i\frac{\sin\xi_i}{\sin\psi_{i+1}}=\ell_i=d_i\frac{\sin\zeta_i}{\sin(4\pi-\gamma_i-\delta_i-(\tau_i+\varphi_i)-(\tau_i-\varphi_i)-(2\pi-\kappa_i)-\psi_{i+1})}\\
	\Rightarrow&c_i\sin\xi_i\sin(4\pi-\gamma-\delta-2\tau_i-(2\pi-\kappa_i)-\psi_{i+1})=d_i\sin\zeta_i\sin\psi_{i+1}
\end{align*}

\begin{enumerate*}
\item Deriving with respect to $\gamma_i$ yields
\begin{align*}
	\text{derivatives w.r.t. }\gamma_i\left\{
	\begin{array}{lll}
	c_i\cos\xi_i\ \xi_i'\sin(4\pi-\gamma_i-\delta_i-2\tau_i-(2\pi-\kappa_i)-\psi_{i+1})\\
	+c_i\sin\xi_i\cos(4\pi-\gamma_i-\delta_i-2\tau_i-(2\pi-\kappa_i)-\psi_{i+1})\ (-1-\psi_{i+1}')\\
	=d_i\sin\zeta_i\ \cos\psi_{i+1}\ \psi_{i+1}'
	\end{array}
	\right.
\end{align*}
We are considering values at $\varphi=0$. Due to symmetry
\begin{align*}
	\xi_i&=\zeta_i
	&c_i&=d_i
	&\psi_{i+1}&=4\pi-\gamma-\delta-2\tau_i-(2\pi-\kappa_i)-\psi_{i+1}=\pi-\frac{\alpha_{i+1}}{2}
\end{align*}
Also, $\xi_i=\gamma_i+\theta$ where $\theta$ is an angle that doesn't change under $\gamma$. Therefore (since we are currently considering derivatives with respect to $\gamma_i$) we have $\xi_i'=1$. Combining all these, we arrive at:
\begin{align*}
	\text{derivatives w.r.t. }\gamma_i\left\{
	\begin{array}{lll}
	c_i\cos\xi_i\ \sin\frac{\alpha_{i+1}}{2}-c_i\sin\xi_i\ \cos\frac{\alpha_{i+1}}{2}\ (-1-\psi_{i+1}')\\
	=-c_i\sin\xi_i\ \cos\frac{\alpha_{i+1}}{2}\ \psi_{i+1}'
	\end{array}
	\right.
\end{align*}
This can be simplified to
\begin{align*}
	\text{derivatives w.r.t. }\gamma_i\left\{
	\begin{array}{lll}
	\cot\xi_i\ \tan\frac{\alpha_{i+1}}{2}-(-1-\psi_{i+1}')=-\psi_{i+1}'
	\end{array}
	\right.
\end{align*}
So we finally arrive at
\begin{align*}
	\frac{\partial\psi_{i+1}}{\partial\gamma_i}=-\frac12\left(\cot\xi_i\tan\frac{\alpha_{i+1}}{2}+1\right)
\end{align*}
Note that at $\varphi=0$, we have $\pi=\xi_i+\sigma_{i+1}+\psi_{i+1}=\xi_i+\sigma_{i+1}+\pi-\frac{\alpha_{i+1}}{2}$ and therefore $\cot\xi_i=\cot(\frac{\alpha_{i+1}}{2}-\sigma_{i+1})$, so we can rewrite this as
\begin{align*}
	\frac{\partial\psi_{i+1}}{\partial\gamma_i}=-\frac12\left(\frac{\tan\frac{\alpha_{i+1}}{2}}{\tan(\frac{\alpha_{i+1}}{2}-\sigma_{i+1})}+1\right)
\end{align*}
Very similar deductions based on derivatives with respect to $\delta_i$ allow us to arrive at
\begin{align*}
	\frac{\partial\psi_{i+1}}{\partial\delta_i}=\frac12\left(\frac{\tan\frac{\alpha_{i+1}}{2}}{\tan(\frac{\alpha_{i+1}}{2}-\sigma_{i+1})}-1\right)
\end{align*}
The claim for $b_i$ follows.

\item We return to the identity based on the law of sines from above, but will now consider derivatives with respect to $\varphi_i$. Note that this time, $c_i$, $d_i$, $\xi_i$ and $\zeta_i$ vary whereas $\gamma_i$ and $\delta_i$ are constant.
\begin{align*}
	\text{derivatives w.r.t. }\varphi_i\left\{
	\begin{array}{lll}
	c_i'\sin\xi_i\sin(4\pi-\gamma_i-\delta_i-2\tau_i-(2\pi-\kappa_i)-\psi_{i+1})\\
	+c_i\cos\xi_i\ \xi_i'\sin(4\pi-\gamma_i-\delta_i-2\tau_i-(2\pi-\kappa_i)-\psi_{i+1})\\
	+c_i\sin\xi_i\cos(4\pi-\gamma_i-\delta_i-2\tau_i-(2\pi-\kappa_i)-\psi_{i+1})\ (-\psi_{i+1}')\\
	=d_i'\sin\zeta_i\sin\psi_{i+1}+d_i\cos\zeta_i\ \zeta_i'\sin\psi_{i+1}+d_i\sin\zeta_i\cos\psi_{i+1}\ \psi_{i+1}'
	\end{array}
	\right.
\end{align*}
We are considering values at $\varphi=0$. Due to symmetry
\begin{align*}
	\xi_i&=\zeta_i
	&c_i&=d_i
	&c_i'&=-d_i'
	&\xi_i'&=-\zeta_i'
	&\psi_{i+1}&=4\pi-\gamma_i-\delta_i-2\tau_i-(2\pi-\kappa_i)-\psi_{i+1}=\pi-\frac{\alpha_{i+1}}{2}
\end{align*}
Combining all these, we arrive at
\begin{align*}
	\text{derivatives w.r.t. }\varphi_i\left\{
	\begin{array}{lll}
	c_i'\sin\xi_i\sin\frac{\alpha_{i+1}}{2}+c_i\cos\xi_i\ \xi_i'\sin\frac{\alpha_{i+1}}{2}+c_i\sin\xi_i\cos\frac{\alpha_{i+1}}{2}\ \psi_{i+1}'\\
	=-c_i'\sin\xi_i\sin\frac{\alpha_{i+1}}{2}-c_i\cos\xi_i\ \xi_i'\sin\frac{\alpha_{i+1}}{2}-c_i\sin\xi_i\cos\frac{\alpha_{i+1}}{2}\ \psi_{i+1}'
	\end{array}
	\right.
\end{align*}
And therefore
\begin{align*}
	\text{derivatives w.r.t. }\varphi_i\left\{
	\begin{array}{lll}
	c_i'\sin\xi_i\sin\frac{\alpha_{i+1}}{2}+c_i\cos\xi_i\ \xi_i'\sin\frac{\alpha_{i+1}}{2}+c_i\sin\xi_i\cos\frac{\alpha_{i+1}}{2}\ \psi_{i+1}'=0
	\end{array}
	\right.
\end{align*}
This can be solved for
\begin{align*}
	\text{derivatives w.r.t. }\varphi_i\left\{
	\begin{array}{lll}
	\psi_{i+1}'=-\tan\frac{\alpha_{i+1}}{2}\left(\frac{c_i'}{c_i}+\cot\xi_i\ \xi_i'\right)
	\end{array}
	\right.
\end{align*}
We now have to find $c_i'/c_i$ and $\xi_i'$. We start with $c_i'/c_i$. Recall that we are still considering derivatives with respect to $\varphi_i$. The law of cosines gives us
\begin{align*}
	\text{derivatives w.r.t. }\varphi_i\left\{
	\begin{array}{lll}
	c_i^2=K^2+L_i^2-2KL_i\cos(\tau_i+\varphi_i)\\
	\Rightarrow2c_ic_i'=2KL_i\sin(\tau_i+\varphi_i)\\
	\Rightarrow\frac{c_i'}{c_i}=\frac{K}{c_i}\frac{L_i}{c_i}\sin(\tau_i+\varphi_i)
	\end{array}
	\right.
\end{align*}
At $\varphi=0$, we have
\begin{align*}
	\varphi_i&=0
	&\frac{K}{c_i}&=\frac{\sin(\gamma_i-\xi_i)}{\sin\tau_i}
	&\frac{L_i}{c_i}&=\frac{\sin(\pi-(\gamma_i-\xi_i)-\tau_i)}{\sin\tau_i}=\frac{\sin(\gamma_i-\xi_i+\tau_i)}{\sin\tau_i}
\end{align*}
We finish with finding $\xi_i'$ (still as a derivative with respect to $\varphi_i$). We invoke the law of tangents for the triangle with sides $c_i$, $K$, $L_i$.
\begin{align*}
	\frac{L_i-K}{L_i+K}=\frac{\tan(\frac12(\pi-(\tau_i+\varphi_i)-(\gamma_i-\xi_i)-(\gamma_i-\xi_i))}{\tan(\frac12(\pi-(\tau_i+\varphi_i)-(\gamma_i-\xi_i)+(\gamma_i-\xi_i))}=\frac{\tan(\frac12(\pi-\tau_i-\varphi_i-2(\gamma_i-\xi_i)))}{\tan(\frac12(\pi-\tau_i-\varphi_i))}
\end{align*}
Note that the left-hand side doesn't change under a change of $\varphi_i$. So if we derive (and subsequently multiply by the denominator), we get
\begin{align*}
	\text{derivatives w.r.t. }\varphi_i\left\{
	\begin{array}{lll}
	0=\sec^2(\frac12(\pi-\tau_i-\varphi_i-2(\gamma_i-\xi_i)))
	(-1+2\xi_i')
	\tan(\frac12(\pi-\tau_i-\varphi_i))\\
	+\tan(\frac12(\pi-\tau_i-\varphi_i-2(\gamma_i-\xi_i)))\sec^2(\frac12(\pi-\tau_i-\varphi_i))
	\end{array}
	\right.
\end{align*}
Elementary trigonometric identities lead us to
\begin{align*}
	\text{derivatives w.r.t. }\varphi_i\left\{
	\begin{array}{lll}
	\displaystyle\xi_i'=\frac12-\frac{\sin(\pi-\tau_i-\varphi_i-2(\gamma_i-\xi_i))}{2\sin(\pi-\tau_i-\varphi_i)}
	\end{array}
	\right.
\end{align*}
At $\varphi=0$, we have $\varphi_i=0$. So we get
\begin{align*}
	\text{derivatives w.r.t. }\varphi_i\left\{
	\begin{array}{lll}
	\displaystyle\xi_i'=\frac12
	-\frac{\sin(\tau_i+2(\gamma_i-\xi_i))}{2\sin\tau_i}
	\end{array}
	\right.
\end{align*}
We now have explicit formulas for $c_i'/c_i$ as well as $\xi_i'$. Combining the above results and using the following two identities at $\varphi=0$ (see figure \ref{fig:sigmaatzero}):
\begin{align*}
	\xi_i&=\pi-\psi_{i+1}-\sigma_{i+1}=\frac{\alpha_{i+1}}{2}-\sigma_{i+1}
	&\pi-(\gamma_i-\xi_i)+\pi/7+\sigma_{i+1}=\pi\ \Rightarrow\ \gamma_i-\xi_i&=\pi/7+\sigma_{i+1}
\end{align*}
we can finally write
\begin{align*}
	&\frac{\partial\psi_{i+1}}{\partial\varphi_i}=-\tan\frac{\alpha_{i+1}}{2}\left(\frac{\sin(\pi/7+\sigma_{i+1})\sin(\pi/7+\sigma_{i+1}+\tau_i)}{\sin\tau_i}+\frac{\frac12
	-\frac{\sin(\tau_i+2\pi/7+2\sigma_{i+1})}{2\sin\tau_i}}{\tan(\frac{\alpha_{i+1}}{2}-\sigma_{i+1})}\right)
\end{align*}
\end{enumerate*}
\end{proof}

\begin{figure}
  \centering
	\includegraphics{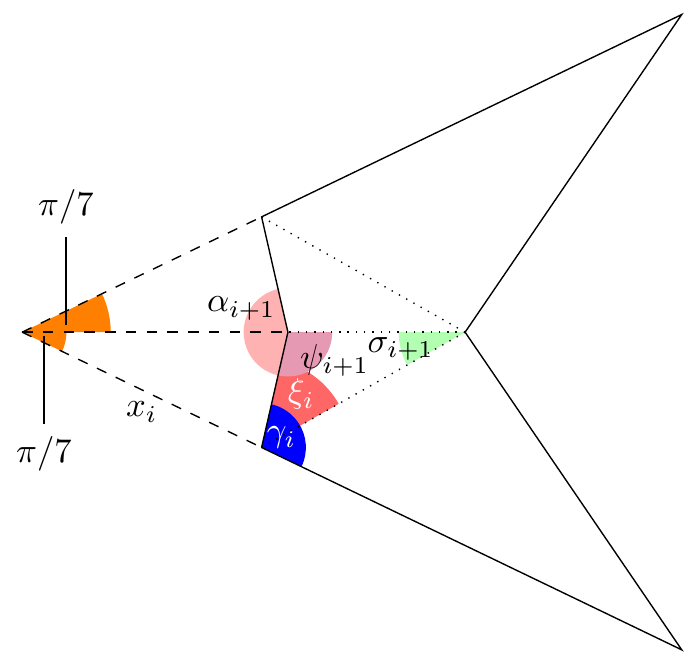}
  \caption{A figure showing the relationship between angles at $\varphi=0$ for $i$ even.}
  \label{fig:sigmaatzero}
\end{figure}

%%%%%%%%%%%%%%%%%%%%%%%%%%%%%%%%%%%%%
\section{Acknowledgements}
%%%%%%%%%%%%%%%%%%%%%%%%%%%%%%%%%%%%%

This research has been partially supported by A. Nabutovsky's NSERC Discovery grant and F. Parsch's Vanier Canada Graduate Scholarship. The authors would like to thank Adam Stinchcombe for valuable help with variable precision arithmetic in MATLAB.


\begin{thebibliography}{Mem15}

\bibitem[AA76]{allardalmgren}
W.~K. Allard and F.~J. Almgren, Jr., \emph{The structure of stationary one
  dimensional varifolds with positive density}, Invent. Math. \textbf{34}
  (1976), no.~2, 83--97.

\bibitem[BK]{becker}
Spencer Becker-Kahn, unpublished.

\bibitem[GL]{guthliok}
Larry Guth and Yevgeny Liokumovich, in preparation.

\bibitem[Hep99]{heppes}
Alad\'{a}r Heppes, \emph{On the partition of the {$2$}-sphere by geodesic
  nets}, Proc. Amer. Math. Soc. \textbf{127} (1999), no.~7, 2163--2165.

\bibitem[HM96]{hassmorgan}
Joel Hass and Frank Morgan, \emph{Geodesic nets on the {$2$}-sphere}, Proc.
  Amer. Math. Soc. \textbf{124} (1996), no.~12, 3843--3850.

\bibitem[IT94]{ivanovtuzhilinbook}
Alexandr~O. Ivanov and Alexei~A. Tuzhilin, \emph{Minimal networks}, CRC Press,
  Boca Raton, FL, 1994.

\bibitem[IT16]{ivanovtuzhilinreview}
Alexander~O. Ivanov and Alexey~A. Tuzhilin, \emph{Minimal networks: a review},
  Advances in dynamical systems and control, Stud. Syst. Decis. Control,
  vol.~69, Springer, 2016, pp.~43--80.

\bibitem[Mem15]{memarian}
Yashar Memarian, \emph{On the maximum number of vertices of critically embedded
  graphs}, International Electronic Journal of Geometry \textbf{8} (2015), 168
  -- 180.

\bibitem[NR07]{NRshapes}
Alexander Nabutovsky and Regina Rotman, \emph{Shapes of geodesic nets}, Geom.
  Topol. \textbf{11} (2007), 1225--1254.

\bibitem[Par18]{Pthree}
Fabian Parsch, \emph{Geodesic nets with three boundary vertices},
  arXiv:1803.03728 (2018), 1--37.

\bibitem[Par19]{Pirreducible}
\bysame, \emph{An example for a nontrivial irreducible geodesic net in the
  plane}, arXiv:1902.07872 (2019), 1--7.

\bibitem[Rot]{Rwide}
Regina Rotman, \emph{Short wide geodesic loops on closed riemannian manifolds},
  in preparation.

\bibitem[Rot11]{Rflowers}
\bysame, \emph{Flowers on {R}iemannian manifolds}, Math. Z. \textbf{269}
  (2011), no.~1-2, 543--554.

\end{thebibliography}
\end{document}